\documentclass{amsart}
\usepackage{amsmath,amssymb,amsthm,newtxtext}
\usepackage{sidecap}
\usepackage{graphics, graphicx}
\usepackage{hyperref}  
\usepackage{mathrsfs}
\usepackage{enumerate}
\usepackage{nicefrac}
\usepackage{tikz}
\usepackage{tikz-cd}
\usepackage{todonotes}
\usepackage{subcaption}

\newtheorem{thm}{Theorem}
\newtheorem{conj}{Conjecture}
\newtheorem{prop}{Proposition}
\newtheorem{lemma}[prop]{Lemma}

\newtheorem{cor}[prop]{Corollary}
\newtheorem{example}[prop]{Example}

\newtheorem*{thm*}{Theorem}
\newtheorem*{alg*}{Algorithm}
\newtheorem*{lemma*}{Lemma}

\theoremstyle{definition}

\theoremstyle{remark}
\newtheorem{rmk}[prop]{Remark}
\newtheorem*{rmk*}{Remark}

\newtheorem*{notation*}{Notation}

\theoremstyle{definition}
\newtheorem{Def}[prop]{Definition}

\renewcommand{\P}{\mathbb{P}}
\newcommand{\R}{\mathbb{R}}
\newcommand{\C}{\mathbb{C}}
\newcommand{\N}{\mathbb{N}}
\newcommand{\Z}{\mathbb{Z}}
\newcommand{\Q}{\mathbb{Q}}

\newcommand{\D}{\mathbb{D}}
\newcommand{\E}{\mathbb{E}}
\newcommand{\U}{\mathbb{U}}

\newcommand{\cD}{\mathcal{D}}

\newcommand{\cS}{\mathcal{S}}

\newcommand{\cM}{\mathcal{M}}

\newcommand{\cT}{\mathcal{T}}

\newcommand{\br}{\mathbf{r}}
\newcommand{\bt}{\mathbf{t}}

\newcommand{\Diag}{\operatorname{Diag}}

\newcommand{\sA}{\mathsf{A}}
\newcommand{\sP}{\mathsf{P}}

\newcommand{\abs}[1]{\left\lvert#1\right\rvert}

\newcommand{\AZ}[2]{\left \langle #1 , #2 \right \rangle}

\title[Areal Weil Heights]{Areal Weil Heights}

\author[Kelley]{Preston Kelley}

\address{Department of Mathematics\\ Oklahoma State University, Stillwater, OK 74078}

\email{preston.kelley@okstate.edu}



\date{\today}

\begin{document}

\begin{abstract}
    In 2008, Pritsker introduced the areal Mahler measure, which is defined using an integral over the unit disk, as opposed to the classical Mahler measure which is defined using an integral over the unit circle. In this paper we introduce areal Weil heights, which generalize the areal Mahler measure to the adelic setting. We use the framework of adelic heights established by Favre and Rivera-Letelier and we construct a $p$-adic analog for the area measure on a disk in $\C$. For areal Weil heights we prove an analog of Kronecker's theorem, which characterizes their small points and essential minima. Furthermore, we determine equidistribution theorems for areal Weil heights. In some cases, they have a unique limiting distribution for small points, while in others there are infinitely many limiting distributions. We conclude with examples. In one of our examples, we determine for which radii $r$ there exist sequences of conjugate sets of algebraic integers which uniformly distribute to the disk $D(0,r) \subset \C$, and we compute the limiting height for such sequences.
\end{abstract}

\maketitle

\tableofcontents

\section{Introduction} \label{introduction section}

\subsection{Background and Motivation}

Let $P(z) = a_n (z-\alpha_1)\cdots (z-\alpha_n) \in \C[z]$ be a nonzero complex polynomial and let $\lambda$ be the normalized arc length measure on the unit circle $S^1 \subset \C$, i.e. $$d\lambda = \left . \frac{d\theta}{2\pi} \right |_{S^1}.$$ The \textit{(logarithmic) Mahler measure} of $P$ is the function 
\begin{gather} \label{def of mahler measure}
m(P) = \int \log \abs{P(z)} d\lambda (z) = \log \abs{a_n} + \sum_{k=1}^n \log^+ \abs{\alpha_k},    
\end{gather}
where the latter equality is a result of Jensen's formula and $\log^+ x = \log \max \{ 1, x\}$ for any $x > 0$. If $\alpha \in \overline{\Q}^\times$ is a nonzero algebraic number, then the Mahler measure of $\alpha$ is $m(\alpha) = m(P_\alpha)$, where $P_\alpha$ is the minimal polynomial for $\alpha$ over $\Z$. The Mahler measure is a height function on the space of algebraic numbers and has seen numerous applications in number theory. 

From (\ref{def of mahler measure}) it is immediate that $m(P) \geq \log \abs{a_n}$. Thus for an algebraic number $\alpha$ we have $m(\alpha) \geq 0$. It is then a consequence of a theorem of Kronecker \cite{kronecker} that for $\alpha \in \overline{\Q}^\times$, $m(\alpha) = 0$ if and only if $\alpha$ is a root of unity. By this observation, it is natural to ask how small can $m(\alpha)$ be for $\alpha \in \overline{\Q}^\times$ which is not a root of unity. This is perhaps the most famous open problem for the Mahler measure, and was posed by Lehmer in \cite{lehmer}. Although Lehmer did not conjecture the solution to this problem in \cite{lehmer}, the following is generally referred to as \textit{Lehmer's Conjecture}.

\begin{conj}[Lehmer's Conjecture] \label{lehmers conjecture}
    There exists $\varepsilon > 0$ such that $m(\alpha) \geq \varepsilon$ for every $\alpha \in \overline{\Q}^\times$ which is not a root of unity.
\end{conj}

The smallest known Mahler measure for a nonzero algebraic integer which is not a root of unity was found by Lehmer in 1933 \cite{lehmer}. For more details on Lehmer's Conjecture \ref{lehmers conjecture}, we refer the reader to \cite{smyth-survey} and \cite{around-the-unit-circle}.

Let $\rho$ be the normalized area measure on the unit disc $\D \subset \C$, i.e. $$d\rho = \left . \frac{dA}{\pi} \right |_\D.$$ In 2008, Pritsker \cite{pritsker-areal} introduced the \textit{areal Mahler measure}, defined by
\begin{gather} \label{expression for areal mahler measure}
    m_\D (P) = \int \log \abs{P(z)} d\rho (z) =  m(P) + \sum_{\abs{\alpha_k} < 1} \frac{\abs{\alpha_k}^2 - 1}{2},
\end{gather}
where the latter equality is proved in \cite[Theorem 1.1]{pritsker-areal}. As a consequence of (\ref{expression for areal mahler measure}), Pritsker \cite[Corollary 1.2]{pritsker-areal} showed that
\begin{gather} \label{pritsker areal inequality}
    -\frac{n}{2} + m(P) \leq m_\D (P) \leq m(P).
\end{gather}
Thus $m_\D$ is very closely related to the classical Mahler measure. In fact, $m_\D$ satisfies a version of Kronecker's theorem: Pritsker proved in \cite[Theorem 1.3]{pritsker-areal} that for $\alpha \in \overline{\Q}^\times$, $m_\D (\alpha) = 0$ if and only if $\alpha$ is a root of unity. Furthermore, Pritsker demonstrated in \cite[Example 1.4]{pritsker-areal} that Lehmer's Conjecture \ref{lehmers conjecture} is false when $m$ is replaced with $m_\D$. This makes $m_\D$ a particularly interesting height to study. Further investigations of the areal Mahler measure can be found in \cite{choi-samuels}, \cite{flammang}, \cite{lalin-roy-1}, \cite{lalin-roy-2}, and \cite{lalin-nair-ringeling-roy}.

Let $h : \overline{\Q} \to \R$ be the \textit{(absolute logarithmic) Weil height} on the space of algebraic numbers, and let $\alpha$ be a nonzero algebraic number of degree $n$ over $\Q$ with minimal polynomial $$P_\alpha (z) = a_n z^n + \cdots + a_0 = a_n (z-\alpha_1) \cdots (z-\alpha_n) \in \Z [z].$$ Then the Mahler measure and the Weil height are related by the formula (see for instance \cite[1.6.6]{bombieri-gubler})
\begin{gather} \label{height and mahler measure relationship}
    h(\alpha) = \frac{1}{n} m(\alpha) = \frac{1}{n} \left ( \log \abs{a_n} + \sum_{k=1}^n \log^+ \abs{\alpha_k} \right ).
\end{gather}
Inspired by the above formula, one may similarly define $h_{\rho_\mathbf{r}} : \overline{\Q} \to \R$ by the following. Let $p_\rho : \C \to [-\infty,\infty]$ be the \textit{potential function} for $\rho$, defined by $$p_\rho (z) = \int \log \abs{z-w} d\rho (w).$$ Then for $\alpha \in \overline{\Q}^\times$, define 
\begin{gather} \label{naive areal height}
    h_{\rho_\mathbf{r}} (\alpha) = \frac{1}{n} m_\D (\alpha) = \frac{1}{n} \left ( \log \abs{a_n} + \sum_{k=1}^n p_\rho (\alpha_k) \right ),
\end{gather}
and furthermore set $h_{\rho_\mathbf{r}} (0) = p_\rho (0) = -\nicefrac{1}{2}$. Combining (\ref{pritsker areal inequality}), (\ref{height and mahler measure relationship}), and (\ref{naive areal height}), we obtain
\begin{gather*}
    -\frac{1}{2} + h(\alpha) \leq h_{\rho_\mathbf{r}} (\alpha) \leq h(\alpha)
\end{gather*}
for any $\alpha \in \overline{\Q}$. This demonstrates that $h_{\rho_\mathbf{r}}$ is a Weil height. The goal of this paper is to study a class of Weil heights which, up to their normalization, generalize $h_{\rho_\mathbf{r}}$. We will refer to heights in this class as \textit{areal Weil heights}. Roughly, areal Weil heights will generalize $h_{\rho_\mathbf{r}}$ by containing ``areal" data at a finite set of places of a number field $K$. At each of the finite places, this data consists of a choice of radius and a local height function which corresponds to this choice. We will construct areal Weil heights using the machinery of adelic heights developed by Favre and Rivera-Letelier in \cite{frl.2}.  After constructing areal Weil heights, we will extend the results of \cite{pritsker-areal} to our setting.

\subsection{Results and Outline of Article} 

We now present the outline of this article and state our main results.

In Section \ref{preliminaries section}, we will cover some technical background that will be used throughout the paper. In particular, we give the definition for adelic measures, adelic heights, and the Arakelov-Zhang pairing.

In Section \ref{definition of areal weil heights} we define areal Weil heights and derive an explicit expression for them. Let $K$ be a number field, let $S \subset M_K$ be a nonempty finite set of places of $K$, and let $\mathbf{r} = (r_v)_{v \in M_K} \in (0,\infty)^S$ be a tuple of radii. For every $v \in M_K$, let $$d_v = \frac{[K_v : \Q_v]}{[K : \Q]}$$ be the local-to-global degree at $v$. We will construct the \textit{areal adelic measure} $\rho_\mathbf{r} = (\rho_{\mathbf{r}, v})_{v \in M_K}$, which is an adelic measure in the sense of Favre and Rivera-Letelier \cite{frl.2}. At the places $v \in S$ such that $v \mid \infty$, the measure $\rho_{\mathbf{r},v}$ is given by the normalized area measure on $r_v \D \subset \C$. At the remaining places $v \in S$, we will construct the $\rho_{\mathbf{r},v}$ so that its potential function has a similar expression to the potential function for the normalized area measure on $r_v \D \subset \C$. We defer the precise definition to Section \ref{definition of areal weil heights}, see Definition \ref{def of areal measure at v}. 

By definition an areal Weil height is an adelic height $h_{\rho_\mathbf{r}}$ associated to an areal adelic measure $\rho_\mathbf{r}$. Our first theorem gives an explicit expression for $h_{\rho_\mathbf{r}} (\alpha)$ for $\alpha \in \P^1 (\overline{K})$. To state our results, we first introduce some helpful notation. For any $r > 0$ we define the function $f_r : [0,\infty) \to \R$ by 
\begin{gather*}
    f_r (x) = \begin{cases}
        \displaystyle \log r - \frac{1}{2} + \frac{x^2}{2r^2} & \text{if } x \leq r, \\
        \log x & \text{if } x \geq r.
    \end{cases}
\end{gather*}
At $v \in S$, the function $f_{r_v} (x)$ will take the role of $\log^+ x$ that appears in (\ref{height and mahler measure relationship}). 

Suppose that $L/K$ is a finite extension. For each $w \in M_L$, we let $$n_w = \frac{[L_w : \Q_w]}{[L:\Q]}$$ be the local-to-global degree of $L/\Q$ at $w$. Let $M_L^S$ denote the set of places $w \in M_L$ such that $w \mid v$ for a place $v \in S$. At each of the places $w \in M_L^S$, let $r_w = r_v$, where $w \mid v$ and $v \in S$. With this notation, we now state Theorem \ref{areal height formula}.

\begin{thm} \label{areal height formula} 
    Let $K$ be a number field, let $S \subset M_K$ be a nonempty finite set of places of $K$, and let $\br  = (r_v)_{v \in S} \in (0,\infty)^S$. Then $\rho_\br$ is an adelic measure in the sense of Favre and Rivera-Letelier \cite{frl.2}. Furthermore, if $\alpha \in L$, where $L/K$ is a finite extension, then 
    \begin{gather*} 
        h_{\rho_\mathbf{r}}(\alpha) = h_{\rho_\mathbf{r}}(\infty) + \sum_{w \in M_L \setminus M_L^S} n_w \log^+ \abs{\alpha}_w + \sum_{w \in M_L^S} n_w f_{r_w} (\abs{\alpha}_w),
    \end{gather*}
    where 
    \begin{gather*}
        h_{\rho_\mathbf{r}}(\infty) = \sum_{v \in S} d_v \left ( \frac{1}{8} - \frac{1}{2} \log r_v \right ).
    \end{gather*}
\end{thm}

In Section \ref{inequalities section} we will prove various inequalities for areal Weil heights, the first of which gives a comparison between areal Weil heights and the classical Weil height.

\begin{prop} \label{areal heights vs weil height}
    Let $\rho_\mathbf{r}$ be an areal adelic measure. Then for any $\alpha \in \P^1 (\overline{K})$, 
    \begin{gather} \label{areal heights vs weil height inequality}
        h(\alpha) + h_{\rho_\mathbf{r}}(\infty) + \sum_{v \in S} d_v \min \{ 0, f_{r_v} (0) \} \leq h_{\rho_\mathbf{r}}(\alpha) \leq h(\alpha) + h_{\rho_\mathbf{r}}(1).
    \end{gather}
    In the upper bound, equality holds when $\alpha$ is a root of unity. In the lower bound, the constant $h_{\rho_\mathbf{r}}(\infty) + \sum_{v \in S} d_v \min \{ 0, f_{r_v} (0) \}$ cannot be improved.
\end{prop}

We also give an analog of Kronecker's theorem for areal Weil heights in Section \ref{inequalities section}. If $\mu$ is an adelic measure, we define the \textit{essential minimum} of $h_\mu$, $$\mathcal{L} (\mu) = \liminf_{\alpha \in \P^1 (\overline{K})} h_\mu (\alpha).$$ In our analog of Kronecker's thereorem for an areal adelic measure $\rho_\mathbf{r}$, we find $\mathcal{L} ({\rho_\mathbf{r}})$ and determine precisely when $h_{\rho_\mathbf{r}}(\alpha) = \mathcal{L} (\rho_\mathbf{r})$. There are two cases, depending on the choice of $\mathbf{r}$. Define 
\begin{gather*}
    \gamma (\mathbf{r}) = \prod_{v \in S} r_v^{[K_v : \Q_v]}. 
\end{gather*}
In (\ref{capacity of E}), we will interpret $\gamma (\mathbf{r})$ as the capacity of a certain adelic set. The analog of Kronecker's theorem in the case where $\gamma (\mathbf{r}) \leq 1$ is given by the following.

\begin{thm} \label{kronecker when gamma (r) < 1}
    Let $\rho_\mathbf{r}$ be an areal adelic measure. For any $\alpha \in \overline{K}^\times$, $h_{\rho_\mathbf{r}}(\alpha) \geq h_{\rho_\mathbf{r}}(\infty)$. Furthermore, suppose $\alpha \in L^\times$, where $L/K$ is a finite extension. Then $h_{\rho_\mathbf{r}}(\alpha) = h_{\rho_\mathbf{r}}(\infty)$ if and only if the following two conditions for $\alpha$ are met:
    \begin{enumerate}[(i)]
        \item If $w \in M_L^S$ then $\abs{\alpha}_w \geq r_w$. \label{lower bound initial case i}
        \item If $w \in M_L \setminus M_L^S$ then $\abs{\alpha}_w \geq 1$. \label{lower bound initial case ii}
    \end{enumerate}
    Finally, if $\gamma (\mathbf{r}) \leq 1$, then $\mathcal{L} (\rho_\mathbf{r}) = h_{\rho_\mathbf{r}}(\infty)$. 
\end{thm}

In the special case where $\gamma (\mathbf{r}) = 1$, we must in fact have equality in (\ref{lower bound initial case i}) and (\ref{lower bound initial case ii}) by the product formula (\ref{product formula}). To give the analog of Kronecker's theorem in the case where $\gamma (\mathbf{r}) > 1$, we utilize another class class of adelic heights, which are associated to adelic measures $\lambda_\mathbf{r} = (\lambda_{\mathbf{r}, v})_{v \in M_K}$, whose definition is given in Definition \ref{lambda heights def}. The utility of these heights in our setting is being able to compare them with areal Weil heights by using the \textit{Arakelov-Zhang pairing} $\AZ{\cdot}{\cdot}$, defined in (\ref{arakelov-zhang pairing def}). Roughly, given two adelic measures $\mu$ and $\nu$ defined over a number field $K$, the Arakelov-Zhang pairing $\AZ{\mu}{\nu}$ gives the distance between $\mu$ and $\nu$. To be more precise, the Arakelov-Zhang pairing is the square of a metric on the space of adelic measures, see \cite[Theorem 1]{fili_ui}. 

\begin{thm} \label{kronecker when gamma (r) > 1}
    Let $\rho_\mathbf{r}$ be an areal adelic measure such that $\gamma (\mathbf{r}) > 1$. Let $\mathbf{t} = c\mathbf{r} = (cr_v)_{v \in S}$, where $c \in (0,1)$ is chosen such that $\gamma (\mathbf{t}) = 1$. Then for any $\alpha \in \overline{K}^\times$, we have
    \begin{gather*}
        h_{\rho_\mathbf{r}} (\alpha) \geq \AZ{\rho_\mathbf{r}}{\lambda_\mathbf{t}} + c^2 (h_{\rho_\mathbf{t}} (\alpha) - h_{\rho_\mathbf{t}} (\infty)),
    \end{gather*}
    with equality if and only if $h_{\rho_\mathbf{t}} (\alpha) = h_{\rho_\mathbf{t}} (\infty)$. Furthermore, $\mathcal{L} (\rho_\mathbf{r}) = \AZ{\rho_\mathbf{r}}{\lambda_\mathbf{t}}$.
\end{thm}

For Theorem \ref{kronecker when gamma (r) > 1}, note that the conditions for $h_{\rho_\mathbf{t}} (\alpha) = h_{\rho_\mathbf{t}} (\infty)$ to hold are covered in Theorem \ref{kronecker when gamma (r) < 1}. We will show in Corollary \ref{spectrum of areal height corollary} that $h_{\rho_\mathbf{r}} (0) < \AZ{\rho_\mathbf{r}}{\lambda_\mathbf{t}}$ and $h_{\rho_\mathbf{r}} (\infty) < \AZ{\rho_\mathbf{r}}{\lambda_\mathbf{t}}$.

In Section \ref{equidistribution section}, we examine equidistribution for sequences $(\alpha_n)_{n=1}^\infty$ of distinct points in $\P^1 (\overline{K})$ such that $h_{\rho_\mathbf{r}}(\alpha_n) \to \mathcal{L}(\rho_\mathbf{r})$ as $n \to \infty$. Our main tool is Favre and Rivera-Letelier's equidistibution theorem, which states that if $h_\mu$ is an adelic measure such that $h_\mu (\alpha_n) \to 0$ as $n \to \infty$, then $[\alpha_n]_v \overset{*}{\to} \mu_v$ at every $v \in M_K$, where convergence is in the weak-* topology (see (\ref{point-mass measure}) for the definition of $[\alpha_n]_v$ and see Theorem \ref{frl thm 2} for the formal statement of the theorem). Since $\mathcal{L} (\rho_\mathbf{r}) > 0$, we can't apply this theorem directly. Insteady, we use the adelic measures $\lambda_\mathbf{t}$, along with the Arakelov-Zhang pairing. In the case when $\gamma (\mathbf{r}) < 1$, we don't have a single limiting distribution, as the following result shows.

\begin{prop} \label{equidistribution for gamma (r) < 1}
    Let $\rho_{\mathbf{r}}$ be an areal adelic measure such that $\gamma (\mathbf{r}) < 1$ and let $S' \subset M_K$ be a finite set of places such that $S \subset S'$. Suppose that $\mathbf{t} = (t_v)_{v \in S} \in (0,\infty)^{S'}$ such that $t_v \geq r_v$ for every $v \in S$, $t_v \geq 1$ for every $v \in S' \setminus S$, and $\gamma (\mathbf{t}) = 1$. Then for any sequence of distinct points $(\alpha_n)_{n=1}^\infty$ in $\P^1 (\overline{K})$ such that $h_{\lambda_\mathbf{t}} (\alpha_n) \to 0$, we have $h_{\rho_\mathbf{r}} (\alpha_n) \to h_{\rho_\mathbf{r}}(\infty)$. In this case, $[\alpha_n]_v \overset{*}{\to} \lambda_{\mathbf{t},v}$ at every $v \in M_K$. 
\end{prop}

This result is nontrivial, since we will show in Proposition \ref{kronecker for lambda heights} that $\mathcal{L} (\lambda_{\mathbf{t}}) = 0$ when $\gamma (\mathbf{t}) = 1$. Alse, as the following example shows, there are infinitely many choices of $\mathbf{t}$ satisfying the conditions in Proposition \ref{equidistribution for gamma (r) < 1}.

\begin{example}
    Let $v_0 \in M_K \setminus S$ and let $S' = S \cup \{v_0\}$. Let $\mathbf{t} = (t_v)_{v \in S'}$ be defined by setting $t_v = r_v$ for $v \in S$ and $$t_{v_0} = \left ( \prod_{v \in S} t_v^{d_v} \right )^{-\frac{1}{d_{v_0}}}.$$ Then $\mathbf{t}$ satisfies the assumptions of Proposition \ref{equidistribution for gamma (r) < 1}. 
\end{example}

In the case where $\gamma (\mathbf{r}) \geq 1$, we have the following result, which can be see as an analog of Bilu's equidistribution theorem \cite[Theorem 1.1]{bilu}. 

\begin{thm} \label{equidistribution for gamma (r) > 1}
    Let $\rho_\mathbf{r}$ be an areal adelic measure such that $\gamma (\mathbf{r}) \geq 1$ and let $\mathbf{t} = c\mathbf{r} = (cr_v)_{v \in S}$, where $c \in (0,1)$ is chosen such that $\gamma (\mathbf{t}) = 1$. If $(\alpha_n)_{n=1}^\infty$ is a sequence of distinct points in $\P^1 (\overline{K})$, then $h_{\rho_\mathbf{r}} (\alpha_n) \to \AZ{\rho_\mathbf{r}}{\lambda_\mathbf{t}} = \mathcal{L} (\rho_\mathbf{r})$ if and only if $h_{\lambda_\mathbf{t}} (\alpha_n) \to 0$. In this case, $[\alpha_n]_v \overset{*}{\to} \lambda_{\mathbf{t},v}$ at every $v \in M_K$. 
\end{thm}

As a corollary, $\lambda_{\mathbf{t}}$ is the unique closest adelic measure to $\rho_\mathbf{r}$ such that its essential minimum is zero.

\begin{cor} \label{peters observation}
    Let $\rho_\mathbf{r}$ be an areal adelic measure such that $\gamma (\mathbf{r}) \geq 1$ and let $\mathbf{t} = c\mathbf{r} = (cr_v)_{v \in S}$, where $c>0$ is chosen such that $\gamma (\mathbf{t}) = 1$. Suppose that $\mu$ is an adelic measure such that $\mathcal{L} (\mu) = 0$. Then $\AZ{\rho_\mathbf{r}}{\mu} \geq \AZ{\rho_\mathbf{r}}{\lambda_\mathbf{t}}$, with equality only when $\mu = \lambda_\mathbf{t}$. 
\end{cor}

Section \ref{examples section} is devoted to examples and has four different subsections. In \ref{lehmers question for areal weil heights}, we will consider an analog of Lehmer's Conjecture \ref{lehmers conjecture} for areal Weil heights $h_{\rho_{\mathbf{r}}}$, where $\mathbf{r} = (1)_{v \in S}$. We show in Proposition \ref{a sequence fails lehmer} that an analog of Lehmer's Conjecture \ref{lehmers conjecture} is false for these heights. In \ref{special class of areal}, we study the following important class of areal Weil heights. Let $S = \{\infty\} \subset M_\Q$ and let $r = r_\infty > 0$. We consider areal Weil heights of the form $h_{\rho_\mathbf{r}}$, where $\mathbf{r} = (r_\infty)$. For \ref{special class of areal} and for the rest of the paper, we use the notation $\rho_r$ in place of $\rho_\mathbf{r}$, since $\mathbf{r}$ only depends on the parameter $r \in (0,\infty)$. In \ref{uniform equidistribution}, we determine for which $r$ a sequence of algebraic integers can uniformly distribute to the measure $\rho_{r,\infty}$, the normalized area measure on $r\D$. We also compute the limiting height for such a sequence. In \ref{computation of pairings}, we give explicit computations of the Arakelov-Zhang pairing of $\rho_r$ with other measures. We conclude with Proposition \ref{minimize pairing prop}, in which we find that $r=2$ minimizes the Arakelov-Zhang pairing of $\rho_r$ and $\mu$, where $\mu$ is the adelic measure associated to the Chebyshev polynomial $T(x) = x^2-2$.

\subsection*{Acknowledgments} The idea for this paper was inspired by a talk about the areal Mahler measure that Igor Pritsker gave at the number theory seminar at Oklahoma State University. I am extremely grateful to my advisor, Paul Fili, who gave me the idea for Definition \ref{def of areal measure at v}. I thank John Doyle for comments on an earlier draft of this paper. Finally, I credit Peter Oberly for the idea of Corollary \ref{peters observation}.

\section{Preliminaries} \label{preliminaries section}

\subsection{Basic Notions}

Let $K$ be a number field, let $M_K$ denote the set of places of $K$, and let $$d_v = \frac{[K_v : \Q_v]}{[K : \Q]}$$ be the local-to-global degree of $K / \Q$. Suppose that $\alpha \in K^\times$. A fundamental tool that we will use several times is the \textit{product formula}:
\begin{gather} \label{product formula}
    \sum_{v \in M_K} d_v \log \abs{\alpha}_v = 0.
\end{gather}

At each $v \in M_K$, let $\C_v$ be a completion of an algebraic closure of $K_v$. We will denote the open disk of radius $r$ centered at $a$ in $\C_v$ by $D_v (a,r)$ and the closed disk by $\overline{D_v} (a,r)$. If $v \mid \infty$, then we identify $\C_v$ with $\C$, and we may write $r\D$ or $D(0,r)$ in place of $D_v (0,r)$.

We let $\sP_v^1$ denote the projective Berkovich line over $\C_v$. If $v \mid \infty$, then we identify $\sP_v^1$ with $\P^1 (\C)$. When $v \nmid \infty$, $\sP_v^1$ is a Hausdorff, locally compact, path connected space which includes $\C_v$ as a dense subspace. For background material on Berkovich spaces we refer the reader to \cite{berkovich}, \cite{baker-rumely.book}, and \cite{benedetto}. Following \cite{baker-rumely.book}, we use the notation $\mathcal{D}_v (a,r)^-$ for the open Berkovich disc of radius $r$ centered at $a$, and $\mathcal{D}_v (a,r)$ for the closed Berkovich disk of radius $r$ centered at $a$. Note that with our notation, if $v \mid \infty$, then we have $D_v (a,r) = \mathcal{D}_v (a,r)^-$ and $\overline{D}_v (a,r) = \mathcal{D}_v (a,r)$. If $a \in \C_v$ and $r >0$, then we let $\zeta_{a,r} \in \sP_v^1$ denote the point in Berkovich space associated to the disk $\overline{D_v} (a,r)$.

\subsection{Adelic Measures and Adelic Heights}

We now briefly review the construction of adelic heights. Define the standard adelic measure $\lambda = (\lambda_v)_{v \in M_K}$ as follows. If $v \mid \infty$, then let $\lambda$ be the normalized arc length measure on $S^1 \subset \C$. If $v$ is finite, let $\lambda_v = \delta_{\zeta_{0,1}}$ be the point mass at the Gauss point $\zeta_{0,1} \in \sP_v^1$. 

\begin{Def} \label{adelic measure def}
    An adelic measure $\mu = (\mu_v)_{v\in M_K}$ defined over $K$ is a sequence of probability measures $\mu_v$ on $\sP_v^1$ subject to the following conditions.
    \begin{enumerate}[(i)]
        \item For all but finitely many places, $\mu_v = \lambda_v$. \label{adelic measure condition i}
        \item At the remaining places, $\mu_v$ has a continuous potential with respect to $\lambda_v$. That is, $\mu_v - \lambda_v = \Delta g_v$ for some continuous $g_v : \sP_v^1 \to \R$, where $\Delta$ is the measure-valued Laplacian on $\sP_v^1$. \label{adelic measure condition ii}
    \end{enumerate}
\end{Def}

Suppose that $\mu_v$ and $\nu_v$ are signed Borel measures on $\sP_v^1$. When it exists, we define the \textit{mutual energy pairing}
\begin{gather*} 
    (\mu_v , \nu_v)_v = \iint_{\sA_v^1 \times \sA_v^1 \setminus \Diag_v} - \log \abs{z-w}_v d\mu_v (w) d\nu_v (z),
\end{gather*}
where $\Diag_v = \{ (z,z) : z \in \sA_v^1\}$ is the diagonal of $\sA_v^1 \times \sA_v^1$ and $\abs{\cdot}_v$ is the Hsia kernel, a natural extension of the $v$-adic absolute value to $\sP_v^1$. Furthermore, if $\mu = (\mu_v)_{v \in M_K}$ and $\nu = (\nu_v)_{v \in M_K}$ are adelic measures defined over $K$, we write $(\mu, \nu)_v$ or $(\mu_v , \nu_v)$ in place of $(\mu_v , \nu_v)_v$ for convenience. 

Now, suppose that $\mu = (\mu_v)_{v \in M_K}$ is an adelic measure defined over $K$, and for each $v \in M_K$ choose a field embedding $\sigma_v : \overline{K} \hookrightarrow \C_v$ which restricts to an isometry on $(K, \abs{\cdot}_v)$. Extend $\sigma_v$ to a map $\P^1 (\overline{K}) \to \P^1 (\C_v)$ by setting $\sigma_v (\infty) = \infty$. Suppose $\alpha \in \P^1 (\overline{K})$ has Galois conjugates $\alpha_1 ,\cdots, \alpha_n$ over $K$. For each $v \in M_K$, we let 
\begin{gather} \label{point-mass measure}
    [\alpha]_v = \frac{1}{n} \sum_{k=1}^n \delta_{\sigma_v (\alpha_k)}
\end{gather}
be the probability measure supported equally on the $K$-conjugates of $\alpha$. Furthermore we let $[\alpha] = ([\alpha]_v)_{v \in M_K}$ denote the sequence formed by the $[\alpha]_v$.

\begin{Def}
    With the above notation, the adelic height for $\mu$ (with respect to $(\sigma_v)_{v \in M_K}$) is the function $h_\mu : \P^1 (\overline{K}) \to \R$ defined by
    \begin{gather*}
        h_\mu (\alpha) = \frac{1}{2} \sum_{v \in M_K} d_v (\mu - [\alpha], \mu - [\alpha])_v.
    \end{gather*}
\end{Def}

\begin{rmk}
    In \cite{frl.2}, Favre and Rivera-Letelier did not make a choice of embedding $\sigma_v$ at each place. However, the choice of $\sigma_v$ can be nontrivial. For instance, if $K= \Q(i)$, $v \mid \infty$, and if $\alpha$ is a root of $z^2+1$, then $[\alpha]_v$ may be either $\delta_i$ or $\delta_{-i}$. Note that for any $v \in M_K$, another choice of embedding $\sigma_v' : \overline{K} \hookrightarrow \C_v$ such that $\sigma_v |_K = \sigma_v' |_K$ will produce the same $[\alpha]_v$. Note that by Theorem \ref{areal height formula}, areal Weil heights do not depend on the choice of embeddings. 
\end{rmk}

\begin{Def} \label{liminf of height}
    Let $\mu$ be an adelic measure defined over a number field $K$. Then the \textit{essential minimum} of $h_\mu$ is 
    \begin{gather*}
        \mathcal{L} (\mu) := \liminf_{\alpha \in \P^1 (\overline{K})} h_\mu (\alpha).
    \end{gather*}
\end{Def}

For later reference, we now state the two main theorems of \cite{frl.2}.

\begin{thm}[\cite{frl.2}, Th\'eor\`em 1] \label{frl thm 1}
    Let $\mu$ be an adelic measure defined over a number field $K$. Then $h_\mu$ is a Weil height and is essentially nonnegative. In other words, there exists a constant $C > 0$ such that $$\abs{h_\mu (\alpha) - h(\alpha)} \leq C$$ for all $\alpha \in \P^1 (\overline{K})$ and $\mathcal{L} (\mu ) \geq 0$.
\end{thm}

\begin{thm}[\cite{frl.2}, Th\'eor\`em 2] \label{frl thm 2}
    Let $\mu = (\mu_v)_{v \in M_K}$ be an adelic measure defined over a number field $K$, and let $(\alpha_n)_{n=1}^\infty \in \P^1 (\overline{K})^\N$ be a sequence of distinct points such that $h_\mu (\alpha_n) \to 0$ as $n \to \infty$. Then $[\alpha_n]_v \overset{*}{\to} \mu_v$ for every $v \in M_K$.
\end{thm}

\subsection{The Arakaelov-Zhang Pairing} One important tool for adelic measures which we will utilize is the Arakelov-Zhang pairing. Suppose that $\mu = (\mu_v)_{v \in M_K}$ and $\nu = (\nu_v)_{v \in M_K}$ are adelic measures defined over a number field $K$. The Arakelov-Zhang pairing of $\mu$ and $\nu$ is defined by
\begin{gather} \label{arakelov-zhang pairing def}
    \AZ{\mu}{\nu} = \frac{1}{2} \sum_{v \in M_K} d_v (\mu - \nu , \mu - \nu )_v. 
\end{gather}
The Arakelov-Zhang pairing is the square of a metric on the space of adelic measures \cite[Theorem 1]{fili_ui}. Note that if $\mu$ is defined over $K$ and $\nu$ is defined over $L$, then both $\mu$ and $\nu$ can be defined over the compositum $KL$ and we may still compute $\AZ{\mu}{\nu}$. One result which we will find particularly helpful is a result of Fili, which extends \cite[Theorem 1]{PST} to the setting of adelic measures.

\begin{thm}[\cite{fili_ui}, Theorem 9] \label{pst theorem}
    Let $\mu$ and $\nu$ be adelic measures defined over a number field $K$. Suppose that $(\alpha_n)_{n=1}^\infty \in \P^1 (\overline{K})^\N$ is a sequence of distinct points. If $h_\mu (\alpha_n) \to 0$ as $n\to\infty$, then $h_\nu (\alpha_n) \to \AZ{\mu}{\nu}$ as $n \to \infty$.
\end{thm}

\subsection{\texorpdfstring{$\gamma (\mathbf{r})$}{gamma r} as Adelic Capacity}

Here we interpret $\gamma (\mathbf{r})$ as the capacity of a certain adelic set. Let $S \subset M_K$ be a finite set of places of $M_K$ and let $\mathbf{r} = (r_v)_{v \in S} \in (0,\infty)^S$ be a tuple of positive real numbers indexed by $S$. We define the quantity 

\begin{gather*}
    \gamma (\mathbf{r}) = \prod_{v \in S} r_v^{[K_v : \Q_v]}.
\end{gather*}

Here, we will interpret $\gamma (\mathbf{r})$ as the capacity of a certain adelic set. We refer the reader to \cite[\textsection 6.6]{baker-rumely.book} for the relevant definitions. Let
\begin{gather} \label{E_r}
    \E_\mathbf{r} = \prod_{v \in M_K} E_{\mathbf{r}, v},
\end{gather}
where 
\begin{gather*} 
    E_{\mathbf{r},v} = \begin{cases}
        \mathcal{D}_v (0,r_v) & \text{if } v \in S,\\
        \mathcal{D}_v (0,1) & \text{if } v \notin S.
    \end{cases}
\end{gather*}

Then $\E_{\mathbf{r}}$ is a $K$-symmetric compact Berkovich adelic set, and its capacity is 
\begin{gather} \label{capacity of E}
    \gamma_\infty (\E_\mathbf{r}) = \prod_{v \in S} r_v^{[K_v : \Q_v]} = \gamma (\mathbf{r}).
\end{gather}

\section{Definition of Areal Weil Heights} \label{definition of areal weil heights}

\subsection{Measure Theory}

We begin with Lemma \ref{construction of measures}, which gives a generalization of defining a measure via a sum of measures. We will use then use this result to define local measures for areal Weil heights. By convention, we will always assume our measures are positive measures, although it is possible to extend our results to more general settings.

\begin{lemma} \label{construction of measures} 
    Let $(X,\cS, \mu)$ be a measure space and let $(Y, \cT)$ be a measurable space. Denote the space of measures on $(Y, \cT)$ by $\cM$ and suppose that $\lambda : X \to \cM$. For convenience, denote $\lambda (x)$ by $\lambda_x$. Finally, let $\Lambda : \cT \to [0,\infty]$ be defined by 
    \begin{gather} \label{construction of measures equation}
        \Lambda (E) = \int_X \lambda_x (E) \, d\mu (x).
    \end{gather}
    
    If for every $E \in \cT$ the function $x \mapsto \lambda_x (E)$ is $\mu$-measurable, then $\Lambda \in \cM$. In this case, for any $\cT$-measurable function $f : Y \to [-\infty, \infty]$ such that $\int_Y f \, d\lambda_x$ is defined for all $x \in X$, the function
    \begin{gather*}
        x \mapsto \int_Y f(y) \, d\lambda_x (y)
    \end{gather*}
    is $\cS$-measurable. Furthermore, if $\int_X f \, d\Lambda$ is well defined, then $$\int_Y f d\Lambda = \int_X \left ( \int_Y f(y) \, d\lambda_x (y) \right ) d\mu (x).$$

    Finally, if $\lambda_x$ is a probability measure for all $x$ and if $\mu$ is also a probability measure, then $\Lambda$ is a probability measure.
\end{lemma}

\begin{proof}
    First note that $\Lambda$ is well defined since $x \mapsto \lambda_x (E)$ is measurable and nonnegative for any $E \in \cT$. Also observe that $$\Lambda (\emptyset) = \int \lambda_x (\emptyset ) \, d\mu (x) = \int 0 \, d\mu (x) = 0.$$ Now, suppose that $\{E_n\}_{n=1}^\infty$ is a disjoint sequence in $\cT$. Then $\sum_{j=1}^n \lambda_x (E_n)$ is a nondecreasing sequence of measurable functions $X \to [0, \infty]$. So using the monotone convergence theorem we compute
    \begin{align*}
        \Lambda \left ( \bigcup_{n=1}^\infty E_n \right ) = & \int \lambda_x \left ( \bigcup_{n=1}^\infty E_n \right ) d\mu (x) \\
        = & \int \sum_{n=1}^\infty \lambda_x (E_n) \, d\mu (x) \\
        = & \sum_{n=1}^\infty \int \lambda_x (E_n) \, d\mu (x) \\
        = & \sum_{n=1}^\infty \Lambda (E_n).
    \end{align*}
    Thus we conclude that $\Lambda \in \cM$.

    To prove the next parts of the theorem, we first consider a nonnegative simple function $s(y) = \sum_{j=1}^n c_j \chi_{E_j} (y)$, where $c_1, \cdots, c_n \in (0,\infty]$ are distinct and $E_1 ,\cdots, E_n \in \cT$ are disjoint. Since $s(y)$ is nonnegative and $\cT$-measurable, it is immediate that $\int_Y s \, d\lambda_x$ is well defined for all $x \in X$ and that $\int_Y s \, d\Lambda$ is well defined. Also, observe that $$\int_Y s(y) \, d\lambda_x = \sum_{j=1}^n c_j \lambda_x (E_j),$$ so that $$x \mapsto \sum_{j=1}^n c_j \lambda_x (E_j)$$ is a linear combination of $\cS$-measurable functions. Thus $x \mapsto \int_Y s d\lambda_x$ is $\cS$-measurable. Furthermore,
    \begin{align*}
        \int_X \left ( \int_Y s(y) \, d\lambda_x (y) \right ) d\mu (x) = & \int_X \left ( \sum_{j=1}^n c_j \lambda_x (E_j) \right ) d\mu (x) \\
        = & \sum_{j=1}^n c_j \int_X \lambda_x (E_j) \, d\mu (x) \\
        = & \sum_{j=1}^n c_j \Lambda (E_j) \\
        = & \sum_{j=1}^n c_j \int_Y \chi_{E_j} (y) \, d \Lambda (y) \\
        = & \int_Y s(y) \, d\Lambda (y).
    \end{align*}
    Thus our theorem holds for nonnegative simple functions.

    We will apply our result on nonnegative simple functions to the general case. Consider a $\cT$-measurable function $f : Y \to [-\infty, \infty]$ such that $\int_Y f \, d\lambda_x$ is well defined for all $x \in X$. Write $f = f^+ - f^-$, where $f^+ = \max \{ 0, f\}$ and $f^- = \max \{0, -f\}$. Note that since $f^+$ and $f^-$ are nonnegative $\cT$-measurable functions, $\int_Y f^+ \, d\lambda_x$ and $\int_Y f^- \, d\lambda_x$ are well defined for all $x \in X$, and that by assumption $\int_Y f \, d\lambda_x = \int_Y f^+  \, d\lambda_x - \int_Y f^- \, d\lambda_x$ is also well defined for all $x \in X$. Now, pick monotone increasing sequences of simple functions $\{s_n^+\}_{n=1}^\infty$ and $\{s_n^-\}_{n=1}^\infty$ such that $s_n^+$ converges to $f^+$ pointwise and $s_n^-$ converges to $f^-$ pointwise. Using the monotone convergence theorem we have for any $x \in X$ that 
    \begin{align*}
        \int_Y f \, d\lambda_x = & \int_Y f^+ \,  d\lambda_x - \int_Y f^- \, d\lambda_x \\
        = & \lim_{n\to\infty} \int_Y s_n^+ \, d\lambda_x - \lim_{n\to\infty} \int_Y s_n^- \, d\lambda_x.
    \end{align*}
    Both $x \mapsto \lim_{n\to\infty} \int_Y s_n^+ \, d\lambda_x$ and $x\mapsto \lim_{n\to\infty} \int_Y s_n^-  \, d\lambda_x$ are pointwise limits of measurable functions, so $x \mapsto \int_Y f \, d\lambda_x$ is measurable. 

    If $\int_Y f \, d\Lambda$ is well defined, then we have, using the monotone convergence theorem, that
    \begin{align*}
        \int_Y f \, d\Lambda = & \int_Y f^+ \, d\Lambda - \int_Y f^- \, d\Lambda \\
        = & \lim_{n\to\infty} \int_Y s_n^+ \, d\Lambda - \lim_{n\to\infty} \int_Y s_n^- \, d\Lambda \\
        = & \lim_{n\to\infty} \int_X \left ( \int_Y s_n^+ \, d\lambda_x \right ) d\mu  -  \lim_{n\to\infty} \int_X \left ( \int_Y s_n^- \, d\lambda_x \right ) d\mu \\
        = & \int_X \left ( \int_Y f^+ \, d\lambda_x \right ) d\mu - \int_X \left ( \int_Y f^- \, d\lambda_x \right ) d\mu \\
        = & \int_X \left (\int_Y f(y) \, d\lambda_x (y) \right ) d\mu (x).
    \end{align*}

    Finally, if we have that $\mu$ is a probability measure and that $\lambda_x$ is a probability measure for all $x \in X$, then 
    \begin{gather*}
        \Lambda (Y) = \int \lambda_x (Y) \, d\mu (x) = \int 1 \, d\mu (x) = 1.\qedhere 
    \end{gather*}
\end{proof}

With the notation from Lemma \ref{construction of measures}, we will write $\Lambda = \int_X \lambda_x d\mu(x)$ for a measure defined by Equation (\ref{construction of measures equation}).

\subsection{The Areal Adelic Measure at a Finite Place}
In this section we construct the local measures for areal Weil heights at finite places. To motivate our construction, we begin with the complex setting. Let $R>0$ and let $\rho_R$ be the normalized area measure on the disc $D(0,R) \subset \C$, i.e. $$d\rho_R = \left . \frac{dA}{\pi R^2} \right |_{D(0,R)}.$$

We can actually define $\rho_R$ using Lemma \ref{construction of measures}. To do this, for any $r > 0$ let $$d\lambda_r : = \left. \frac{d\theta}{2\pi} \right |_{\partial D(0,r)}$$ be the normalized angular measure on $\partial D(0,r)$ and let $$d\mu_R : = \left. \frac{2rdr}{R^2} \right |_{D(0,R)}.$$ By Lemma \ref{construction of measures} we have that 
\begin{gather} \label{Lambda for rho_R}
    \Lambda : = \int_0^R \lambda_r \, d\mu_R (r)  
\end{gather}
is a probability measure. If $E \subset \C$ is a Borel subset, then
\begin{align*}
    \Lambda (E) = & \int_0^R \lambda_r (E) \, d\mu_R (r) \\
    = & \frac{2}{R^2} \int_0^R \left ( \frac{1}{2\pi} \int_0^{2\pi} \chi_E (re^{i\theta}) \, d\theta \right ) rdr \\
    = & \rho_R (E).
\end{align*}
Thus we have $\Lambda = \rho_R$.

We now consider the nonarchimedean setting. Let $K$ be a number field and let $v \in M_K$ be a finite place of $K$. We aim to define a Borel probability measure $\rho_{R,v}$ on $\sP_v^1$ which is analogous to $\rho_R$. We note that $D_v (0,1)$ is not compact, so it does not have a Haar measure. Instead of this idea, we construct $\rho_{R,v}$ so that its potential is similar to the potential of $\rho_R$. In the complex case, we defined $\lambda_r$ as the normalized Lebesgue measure on $\partial D(0,r)$. For the nonarchimedean case, we have that the boundary of the open Berkovich disc $\cD_v (0,r)^- \subset \sP_v^1$ is the point $\zeta_{0,r}$.  So we define $\lambda_{r,v} = \delta_{\zeta_{0,r}}$, the point mass on $\zeta_{0,r} = \partial \cD_v (0,r)^-$. Note that the potential functions of $\lambda_r$ and $\lambda_{r,v}$ have similar expressions, see Equations (\ref{potential of lambda v infinite}) and (\ref{potential of lambda v finite}). Inspired by (\ref{Lambda for rho_R}), using Lemma \ref{construction of measures} we define
\begin{gather*}
    \rho_{R,v} = \int_0^R \lambda_{r,v} \, d\mu_R (r).
\end{gather*}
We claim that $\rho_{R,v}$ is a Borel probability on $\sP_v^1$. We wish to apply Lemma \ref{construction of measures}, so we need to show that for any Borel measurable $E \subset \sP_v^1$, the function $r \mapsto \lambda_{r,v} (E)$ is measurable. More generally, we have the following.

\begin{prop} \label{areal case is measurable} 
    Suppose that $a \in \C_v$ with $\abs{a}_v < R$ and that $E \subset \sP_v^1$ is measurable. Then the function $\varphi : [0,R] \to \{0,1\}$ defined by $\varphi (r) = \delta_{\zeta_{a,r}} (E)$ is measurable.
\end{prop}

\begin{proof}
    Let $\iota : [0,R] \to \sP_v^1$ be defined by $\iota (r) = \zeta_{a,r}$. Then $\iota$ is a continuous embedding of $[0,R]$ into $\sP_v^1$. Let $\chi_E$ denote the characteristic function for $E$, which is measurable since $E$ is measurable. Observe that $\chi_E (\zeta_{a,r}) = \delta_{\zeta_{a,r}} (E)$. So $\varphi = \chi_E \circ \iota$ is a composition of measurable functions. 
\end{proof}

Proposition \ref{areal case is measurable} implies that $\rho_{R,v}$ satisfies the conditions of Lemma \ref{construction of measures}, so we have the following.

\begin{cor} \label{rho is a probability}
    With the notation as above, $\rho_{R, v}$ is a positive Borel probability measure on $\sP_v^1$.
\end{cor}

 Corollary \ref{rho is a probability} allows us to make the following definition. 

\begin{Def} \label{def of areal measure at v}
    Let $K$ be a number field, let $v \in M_K$, and let $R>0$. The \textit{areal measure} at $v$ with radius $R>0$ is $$\rho_{R,v} = \int_0^R \lambda_{r,v} \, d\mu_R (r),$$ where $d\lambda_{r,v} = \left . \frac{d\theta}{2\pi} \right |_{\partial D(0,r)}$ when $v \mid \infty$ and $\lambda_{r,v} = \delta_{\zeta_{0,r}}$ when $v$ is finite.
\end{Def}

\begin{rmk} \label{construction using laplacian}
    Alternatively, we could have constructed $\rho_{R,v}$ by taking the measure-valued Laplacian of a suitable continuous function, as in Condition (\ref{adelic measure condition ii}) of Definition \ref{adelic measure def}. However, we prefer our construction, which gives us an explicit description of $\rho_{R,v}$.
\end{rmk}

\begin{rmk} \label{change of base point} \label{check math here}
    If $v$ is finite, then note that for any $a \in D_v (0,R) \subset \C_v$ we have $D_v(0,R) = D_v(a,R)$. For this choice of $a$, we may define the measure $$\upsilon_{a,r} = \int_0^R \delta_{\zeta_{a,r}} d\mu_R (r)$$ using Lemma \ref{construction of measures} and Proposition \ref{areal case is measurable}. We begin by noting that if $a \neq 0$ then $\upsilon_{a,r} \neq \rho_{r,v}$. However, they are related by $$\upsilon_{a,r} = \varphi_*\rho_{r,v},$$ where $\varphi (z) = z + a$. Although there is not necessarily a canonical choice of center $a \in D_v(0,R)$, one important property of $\rho_{r,v}$ is that its potential function $p_{\rho_{r,v}} (z)$ satisfies $$p_{\rho_{r,v}} (z) = p_{\rho_{r,v}} (w)$$ whenever $\abs{z}_v = \abs{w}_v$. This is not the case when $a \neq 0$. For instance, if $v \nmid 2$ then $$p_{\upsilon_{a,r}} (a) = p_{\rho_{r,v}} (2a) = p_{\rho_{r,v}} (a) \neq p_{\rho_{r,v}} (0) = p_{\upsilon_{a,r}} (-a).$$
\end{rmk}

\subsection{Potentials and Energy} We now study the areal measures defined in the previous section. Let $K$ be a number field and $v \in M_K$. Recall that for a Borel probability measure $\mu$ on $\sP_v^1$, we define the potential function $p_\mu : \sP_v^1 \to [-\infty, \infty]$ by $$p_\mu (z) = \int_{\sA_v^1} \log \abs{z-w}_v d\mu (z).$$ Furthermore, we define the \textit{energy} of $\mu$ as $$(\mu , \mu)_v = \iint_{\sA_v^1 \times \sA_v^1 \setminus \Diag_v} -\log \abs{z-w} d\mu (w) d\mu (z).$$   
We have the following proposition.

\begin{prop} \label{potential function}
    Let $K$ be a number field, let $v$ be a place of $K$, and let $R>0$. Then the potential function for the areal measure $\rho = \rho_{R,v}$ is given by
    \begin{gather} \label{rho potential}
        p_\rho (z) = \begin{cases}
            \displaystyle \log R - \frac{1}{2} + \frac{\abs{z}_v^2}{2R^2} & \text{if } \abs{z}_v \leq R, \\
            \log \abs{z}_v & \text{if } \abs{z}_v \geq R
        \end{cases}
    \end{gather}
    and its energy is
    \begin{gather*}
        (\rho , \rho )_v = \frac{1}{4} - \log R.
    \end{gather*}
\end{prop}

\begin{proof}
Our computations depend on whether $v$ is finite or infinite, so we first suppose that $v \mid \infty$ and we identify $\sP_v^1$ with $\P^1 (\C)$. Using Jensen's formula, one may demonstrate that 
\begin{gather} \label{potential of lambda v infinite}
    p_{\lambda_{r,v}}(z) = \frac{1}{2\pi} \int_0^{2\pi} \log \abs{z-re^{i\theta}}_v d\theta = \log \max \{ r, \abs{z}_v \}
\end{gather}
for any $r > 0$. So 
\begin{align}
\begin{split} \label{complex case rho potential}
    p_\rho (z) = & \frac{1}{\pi R^2} \int_0^R \int_0^{2\pi} \log \abs{z-re^{i\theta}} d\theta \, rdr \\
    = & \frac{2}{R^2} \int_0^R \log \max \{ r, \abs{z}_v \} \, r dr.
\end{split}
\end{align}
By considering cases depending on whether or not $\abs{z}_v \leq R$, one shows that this evaluates to (\ref{rho potential}). Next we compute
\begin{align*}
    (\rho , \rho)_v = & \int_{D(0,R)} -p_\rho (z) \, d\rho (z) \\
    = & \frac{1}{\pi R^2} \int_{D(0,R)} \left ( -\log R + \frac{1}{2} - \frac{\abs{z}^2}{2R^2} \right ) dA(z) \\
    = & \frac{1}{4} - \log R.
\end{align*}
Thus the result holds when $v \mid \infty$. Now suppose $v$ is finite. We begin by computing that for any $r>0$,
\begin{align} \label{potential of lambda v finite}
    \begin{split}
        p_{\lambda_{r,v}} (z) = & \int_{\sA_v^1} \log \abs{z-w}_v d\lambda_{r,v} (w) \\
        = & \log \abs{z-\zeta_{0,r}}_v \\
        = & \log \max \{ r, \abs{z}_v \}.
    \end{split}
\end{align}
By Proposition 6.11 in \cite{baker-rumely.book}, $p_\rho (z)$ is defined for any $z \in \sP_v^1$. Thus we may apply Lemma \ref{construction of measures} to compute $p_\rho (z)$. We have
\begin{align*}
    p_\rho (z) = & \int \log \abs{z-w}_v d\rho (w) \\
    = & \frac{2}{R^2} \int_0^R \left ( \int \log \abs{z-w}_v d\lambda_{r,v} (w) \right ) rdr \\
    = & \frac{2}{R^2} \int_0^R \log \max \{r, \abs{z}_v \} \, rdr.
\end{align*}
Similarly to (\ref{complex case rho potential}), this evaluates to  (\ref{rho potential}).

Next, note that $\int -p_\rho (z) \, d\lambda_{r,v} (z) = -p_\rho (\zeta_{0,r})$ is defined for any $r>0$. Furthermore, one may verify that $-p_\rho (z) \leq -p_\rho (0)$ for any $z \in \sP_v^1$, which implies that $$\int -p_\rho (z) d\rho (z) \leq -p_\rho (0) < \infty.$$ Combining this with the fact that $-p_\rho (z)$ is lower semicontinuous on $\sA_v^1$ (Proposition 6.11 in \cite{baker-rumely.book}), we know $\int -p_\rho (z) \, d\rho (z)$ exists. So we may apply Lemma \ref{construction of measures} to compute
\begin{align*}
    (\rho ,\rho)_v = & \int_{\mathcal{D}_v (0,R)} - p_\rho (z) \, d\rho (z) \\
    = & \frac{2}{R^2} \int_0^R  \int \left ( - \log R + \frac{1}{2} - \frac{\abs{z}_v^2}{2R^2} \right ) d\lambda_{r,v} (z) \, rdr \\
    = & \frac{2}{R^2} \int_0^R \left (-\log R + \frac{1}{2} - \frac{r^2}{2R^2} \right ) rdr\\
    = & \frac{1}{4} - \log R. \qedhere
\end{align*}
\end{proof}

Recall from Section \ref{introduction section} that we define the function $f_R : [0,\infty) \to \R$ by 
\begin{gather} \label{f_r function}
    f_R (x) = \begin{cases}
        \displaystyle \log R - \frac{1}{2} + \frac{x^2}{2R^2} & \text{if } x \leq R, \\
        \log x & \text{if } x \geq R.
    \end{cases}
\end{gather}
Then by Proposition \ref{potential function} we have that $p_{\rho_{R,v}} (z) = f_R (\abs{z}_v)$. Since both $f_R$ and $\abs{z}_v$ are continuous, we have the following.

\begin{cor} \label{potentials are continuous}
    Let $K$ be a number field, let $v$ be a place of $K$, and let $R>0$. Then $p_{\rho_{R,v}}$ is continuous.
\end{cor}

\begin{figure} \label{graphs of f_R}
\centering
\begin{subfigure}{0.24\textwidth}
    \centering
    \includegraphics[width=.95\linewidth]{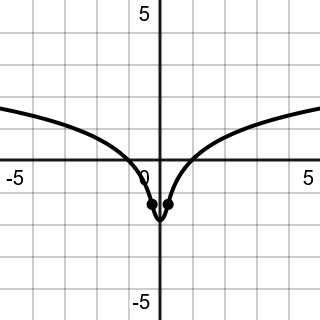}  
    \caption{$y=f_{\nicefrac{1}{4}}(\abs{x})$}
    \label{graph with radius 1/4}
\end{subfigure}
\begin{subfigure}{0.24\textwidth}
    \centering
    \includegraphics[width=.95\linewidth]{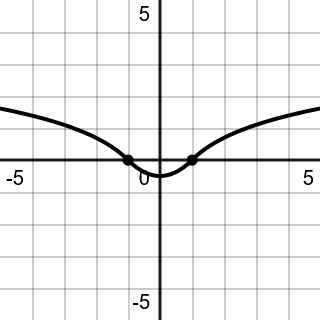} 
    \caption{$y=f_1 (\abs{x})$}
    \label{graph with radius 1}
\end{subfigure}
\begin{subfigure}{0.24\textwidth}
    \centering
    \includegraphics[width=.95\linewidth]{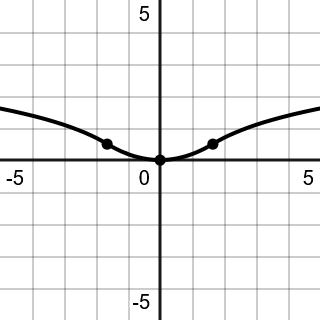}  
    \caption{$y=f_{e^{\nicefrac{1}{2}}} (\abs{x})$}
    \label{graph with radius e^.5}
\end{subfigure}
\begin{subfigure}{0.24\textwidth}
    \centering
    \includegraphics[width=.95\linewidth]{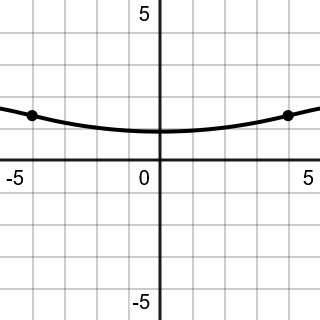} 
    \caption{$y=f_4 (\abs{x})$}
    \label{graph with radius 4}
\end{subfigure}
\caption{Graphs of $y=f_R (\abs{x})$ for various $R$. The points $(\pm R, f_R (\abs{R}))$ are indicated in each graph. In Figure \ref{graph with radius e^.5} the graph's intersection with the origin is indicated.}
\end{figure}

As another corollary, we have that $p_{\rho_{R,v}}$ is radial.

\begin{cor}
    Suppose that $z, w \in \sP_v^1$ such that $\abs{z}_v = \abs{w}_v$. Then we have $p_{\rho_{R,v}} (z) = p_{\rho_{R,v}} (w)$. 
\end{cor}

\subsection{The Areal Weil Height} 

Finally we define areal adelic measures and areal Weil heights.

\begin{Def} \label{definition of areal adelic measure}
    Let $S \subset M_K$ be a finite set of places of a number field $K$ and let $\br = (r_v)_{v \in S} \in (0,\infty)^S$ be a tuple of positive numbers indexed by $S$. Then the areal adelic measure for $\br$ is the sequence $\rho_\br = (\rho_{\br , v})_{v \in M_K}$, where $$\rho_{\br , v} = \begin{cases}
        \rho_{r_v, v} & \text{if } v \in S \\
        \lambda_v & \text{if } v \notin S.
    \end{cases}$$
    An areal Weil height is an adelic height of the form $h_{\rho_\mathbf{r}}$, where $\rho_\mathbf{r}$ is an areal adelic measure.
\end{Def}

We now recall the setup of Theorem \ref{areal height formula}. For every $v \in M_K$, we let $$d_v = \frac{[K_v : \Q_v]}{[K : \Q]}$$ be the local-to-global degree at $v$. Keep the notation from Definition \ref{definition of areal adelic measure} and suppose that $\alpha \in \overline{K}$. Let $L/K$ be a finite extension containing $\alpha$. For each $w \in M_L$ let $$n_w = \frac{[L_w : \Q_w]}{[L:\Q]}$$ be the local-to-global degree of $L/\Q$ at $w \in M_L$. Furthermore let $M_L^S$ be the set of places $w \in M_L$ which lie above some $v \in S$. At every $w \in M_L^S$, let $r_w = r_v$, where $v \in M_K$ and $w \mid v$. We will prove that 
    \begin{gather} \label{formula of areal height}
        h_{\rho_\mathbf{r}} (\alpha) = h_{\rho_\mathbf{r}} (\infty) + \sum_{w \in M_L \setminus M_L^S} n_w \log^+ \abs{\alpha}_w + \sum_{w \in M_L^S} n_w f_{r_w} (\abs{\alpha}_w),
    \end{gather}
where 
    \begin{gather*} 
        h_{\rho_\mathbf{r}} (\infty) = \sum_{v \in S} d_v \left ( \frac{1}{8} - \frac{1}{2} \log r_v \right ).
    \end{gather*}
Furthermore, we will use this notation throughout the rest of the paper without further comment.

\begin{proof}[Proof of Theorem \ref{areal height formula}]
    The fact that $\rho_\br$ is an adelic measure follows from the finiteness of $S$, the fact that $\rho_{r_v, v}$ is a Borel probability on $\sP_v^1$ for any $v \in S$, and from Corollary \ref{potentials are continuous}. 
    
    Let $\alpha_1, \cdots, \alpha_n$ be the distinct conjugates of $\alpha$ over $K$ and let $(\sigma_v)_{v \in M_K}$ be a sequence of embeddings $\sigma_v : \overline{K} \to \C_v$ which restrict to isometries on $(K,\abs{\cdot}_v)$. Define $[\alpha] = ([\alpha]_v)_{v \in M_K}$, where $$[\alpha]_v = \frac{1}{n} \sum_{k=1}^n \delta_{\sigma_v (\alpha_k)}.$$ Using Lemme 5.3 of \cite{frl.2} we may write
    \begin{gather} \label{FRL height formula}
        h_{\rho_\mathbf{r}} (\alpha) = h_{\rho_\mathbf{r}} (\infty) - \sum_{v \in M_K} d_v \, (\rho_\mathbf{r} , [\alpha])_v,
    \end{gather}
    where 
    \begin{gather*}
        \begin{split}
            h_{\rho_\mathbf{r}} (\infty) = \frac{1}{2} \sum_{v \in M_K} d_v \, (\rho_\mathbf{r} , \rho_\mathbf{r} )_v.
        \end{split}
    \end{gather*}
    Proposition \ref{potential function} then implies
   \begin{gather*}
       h_{\rho_\mathbf{r}} (\infty) = \sum_{v \in S} d_v \left ( \frac{1}{8} - \frac{1}{2} \log r_v \right ).
   \end{gather*}
   On the other hand, for $v \in S$ we use Proposition \ref{potential function} to compute
    \begin{gather} \label{local height}
    \begin{split}
        -(\rho_\mathbf{r} , [\alpha])_v = & \iint_{\sA_v^1 \times \sA_v^1 \setminus \Diag_v} \log \abs{z-w}_v d\rho_{\br,v} (z) d[\alpha]_v (w) \\
        = & \frac{1}{n} \sum_{k=1}^n \int_{\sA_v^1} \log \abs{z-\sigma_v \alpha_k} d\rho_{\br,v} (z) \\
        = & \frac{1}{n} \sum_{k=1}^n p_{\rho_{\br,v}} ( \sigma_v \alpha_k) \\
        = & \frac{1}{n} \sum_{k=1}^n f_{r_v} (\abs{\sigma_v \alpha_k}_v) \\
        = & \sum_{\substack{w \in M_L\\w\mid v}} \frac{n_w}{d_v} f_{r_v} (\abs{\alpha}_w).
    \end{split}
    \end{gather}
    For $v \in M_K \setminus S$, we have 
    \begin{gather} \label{local height part 2}
        -(\rho_\mathbf{r}, [\alpha])_v = \frac{1}{n} \sum_{k=1}^n p_{\lambda_v} (\sigma_v \alpha_k) = \frac{1}{n} \sum_{k=1}^n \log^+ \abs{\sigma_v \alpha_k}_v =\sum_{\substack{w \in M_L\\w\mid v}} \frac{n_w}{d_v} \log^+ \abs{\alpha}_w.
    \end{gather}
    Plugging (\ref{local height}) and (\ref{local height part 2}) into (\ref{FRL height formula}), we obtain the desired expression for $h_{\rho_\mathbf{r}} (\alpha)$. 
\end{proof}

\begin{rmk} \label{another expression for areal weil heights}
    In some contexts, it is more convenient to write areal Weil heights using embeddings of conjugates rather than places of an extension of $K$. From the proof of Theorem \ref{areal height formula}, we see that we may write 
    \begin{gather} \label{another equation for areal weil heights}
        h_{\rho_\mathbf{r}} (\alpha) = h_{\rho_\mathbf{r}} (\infty) + \sum_{v \in M_K \setminus S} \frac{d_v}{n} \sum_{k=1}^n \log^+ \abs{\sigma_v \alpha_k}_v + \sum_{v \in S} \frac{d_v}{n} \sum_{k=1}^n f_{r_v} (\abs{\sigma_v \alpha_k}_v ).
    \end{gather}
    Furthermore, we have
    \begin{multline} \label{another equation for areal weil heights S sum}
        \sum_{v \in S} \frac{d_v}{n} \sum_{k=1}^n f_{r_v} (\abs{\sigma_v \alpha_k}_v ) = \\\sum_{v \in S} \frac{d_v}{n} \left (  \sum_{\abs{\sigma_v \alpha_k}_v < r_v} \log r_v - \frac{1}{2} + \frac{\abs{\sigma_v \alpha_k}_v^2}{2r_v^2} + \sum_{\abs{\sigma_v \alpha_k}_v \geq r_v} \log \abs{\sigma_v \alpha_k}_v \right ).
    \end{multline}
    Since the above expressions do not depend on the choice of embeddings $(\sigma_v)_{v \in M_K}$, it is also reasonable to abuse notation and remove all instances of ``$\sigma_v$" from (\ref{another equation for areal weil heights}) and (\ref{another equation for areal weil heights S sum}). We will follow this convention in the proof of Theorem \ref{equidistribution for gamma (r) > 1}.
\end{rmk}

For the following examples we keep the same notation from Theorem \ref{areal height formula}.

\begin{example} \label{height of 0}
The areal Weil height at zero is given by
    \begin{align*}
        h_{\rho_\mathbf{r}} (0) = h_{\rho_\mathbf{r}} (\infty) + \sum_{v \in S} d_v \left ( \log r_v - \frac{1}{2} \right ) = \sum_{v \in S} d_v \left ( \frac{1}{2} \log r_v -\frac{3}{8} \right ).
    \end{align*}
Note that $h_{\rho_\mathbf{r}} (0) < 0$ when $r_v < e^{\nicefrac{3}{4}}$ for every $v \in S$. 
\end{example}

\begin{example} \label{height of zeta}
    Suppose that $\zeta$ is a root of unity. Then
    \begin{align} \label{pairing of rho and lambda}
    \begin{split}
        h_{\rho_\mathbf{r}} (\zeta) =  & h_{\rho_\mathbf{r}} (\infty ) + \sum_{\substack{v \in S \\ r_v > 1}} d_v \left ( \log r_v - \frac{1}{2} + \frac{1}{2r_v^2} \right ) \\
        = & \sum_{\substack{v \in S\\r_v \leq 1}} d_v \left ( \frac{1}{8} - \frac{1}{2} \log r_v \right ) + \sum_{\substack{v \in S \\ r_v > 1}} d_v \left ( \frac{1}{2} \log r_v - \frac{3}{8} + \frac{1}{2r_v^2} \right ). 
    \end{split}
    \end{align}
    From the above expression we see that $h_{\rho_\mathbf{r}} (\zeta) > 0$. Also, note that it is a consequence of Theorem \ref{pst theorem} that (\ref{pairing of rho and lambda}) is equal to $\AZ{\rho_\mathbf{r}}{\lambda}$. 
\end{example}

\section{Inequalities} \label{inequalities section}

\subsection{Basic Inequalities}

For $r>0$, recall that we define $f_r : [0,\infty) \to \R$ by (\ref{f_r function}). Our results will make use of several properties of $f_r$, some of which we record here for convenience.

\begin{lemma} \label{f_r properties}
    The function $f_r (x)$ is once differentiable and strictly increasing in $x$. Furthermore, the following inequalities hold for $f_r$.
    \begin{enumerate}[(a)]
        \item $f_r(x) \geq \log x$, with equality if and only if $x \geq r$. Furthermore, $f_r (x) - \log x$ is strictly decreasing in $(0,r_v)$.  \label{f_r properties a}
        \item $f_r(x) \leq \log^+ x + f_r (1)$. If $r \leq 1$, then equality holds if and only if $x \geq 1$. If $r > 1$, then equality holds if and only if $x=1$. \label{f_r properties b}
        \item $f_r (x) \geq \log^+ x + \min \{ 0, f_r (0) \}$. If $r < e^{\nicefrac{1}{2}}$, then equality holds if and only if $x=0$. If $r=e^{\nicefrac{1}{2}}$, then equality holds if and only if $x=0$ or $x \geq r$. If $r>e^{\nicefrac{1}{2}}$ then equality holds if and only if $x \geq r$. \label{f_r properties c}
    \end{enumerate}
\end{lemma}

We omit the proof of Lemma \ref{f_r properties}, which only requires elementary calculus. We now prove our comparison of $h_{\rho_\mathbf{r}}$ with the classical Weil height.

\begin{proof}[Proof of Proposition \ref{areal heights vs weil height}]
    We first note that when $\alpha = \infty$, (\ref{areal heights vs weil height inequality}) becomes $$h_{\rho_\mathbf{r}} (\infty) + \sum_{v \in S} d_v \min \{0,f_{r_v}(0)\} \leq h_{\rho_\mathbf{r}} (\infty) \leq  h_{\rho_\mathbf{r}} (\infty) + \sum_{v \in S} d_v f_{r_v} (1).$$ So it is immediate that the lower bound holds, and the upper bound follows from the fact that $f_{r_v} (1) \geq 0$ by Lemma \ref{f_r properties} (\ref{f_r properties a}).

    Now suppose that $\alpha \in \overline{K}$ and pick a finite extension $L / K$ such that $\alpha \in L$. We begin by considering the upper bound. Using Lemma \ref{f_r properties} (\ref{f_r properties b}), we have
    \begin{align*}
        \sum_{w \in M_L^S} n_w f_{r_w} (\abs{\alpha}_w) \leq & \sum_{w \in M_L^S} n_w (\log^+ \abs{\alpha}_w + f_{r_w} (1)) \\
        = & \sum_{w \in M_L^S} n_w \log^+ \abs{\alpha}_w + \sum_{v \in S} d_v f_{r_v} (1).
    \end{align*}
    By combining the above inequality with (\ref{formula of areal height}), we see that the upper bound of (\ref{areal heights vs weil height inequality}) holds. If $\alpha$ is a root of unity, then $h(\alpha) = 0$ and equality in the upper bound holds.

    Now let's consider the lower bound. By Lemma \ref{f_r properties} (\ref{f_r properties c}) we have 
    \begin{align*}
        \sum_{w \in M_L^S} n_w f_{r_w} (\abs{\alpha}_w) \geq & \sum_{w \in M_L^S} n_w (\log^+ \abs{\alpha}_w + \min \{ 0, f_{r_w}(0) \} ) \\
        = & \sum_{w \in M_L^S} n_w \log^+ \abs{\alpha}_w + \sum_{v \in S} d_v \min \{ 0, f_{r_v} (0) \}.
    \end{align*}
    By combining the above inequality with (\ref{formula of areal height}), we obtain the lower bound of (\ref{areal heights vs weil height inequality}).
    To show that this bound cannot be improved, pick $0 < \varepsilon < \min_{v \in S} r_v$. Find $\alpha \in K$ such that 
    \begin{enumerate}[(i)]
        \item If $v \in S$ such that $r_v < e^{\nicefrac{1}{2}}$, then $\abs{\alpha}_v < \varepsilon$.
        \item If $v \in S$ such that $r_v \geq e^{\nicefrac{1}{2}}$, then $\abs{\alpha}_v \geq r_v$.
    \end{enumerate}
    There are no conditions at the remaining places $v \notin S$. Such an $\alpha$ can be found using a simultaneous approximation theorem, see for instance \cite[Theorem 3.1]{cassels}. With such a choice of $\alpha$, we have
    \begin{align*}
        0 \leq & h_{\rho_\mathbf{r}} (\alpha) - h(\alpha) - h_{\rho_\mathbf{r}} (\infty) - \sum_{v \in S} d_v \min \{ 0, f_{r_v} (0) \} \\= & \sum_{\substack{v \in S\\r_v < e^{\nicefrac{1}{2}}}} d_v (f_{r_v} (\abs{\alpha}_v) - f_{r_v} (0)) \\
        < & \sum_{\substack{v \in S\\r_v < e^{\nicefrac{1}{2}}}} d_v (f_{r_v} (\varepsilon) - f_{r_v} (0)).
    \end{align*}
    Since $f_{r_v}$ is continuous for every $v \in S$ and since $S$ is finite, as $\varepsilon \to 0$, $$h_{\rho_\mathbf{r}} (\alpha) - h (\alpha) - h_{\rho_\mathbf{r}} (\infty) - \sum_{v \in S} d_v \min \{0 , f_{r_v} (0) \} \to 0.$$ Thus the lower bound of (\ref{areal heights vs weil height inequality}) cannot be improved.
\end{proof}

\subsection{Lambda Heights}

We now introduce another class of adelic heights which will play a key role in determining lower bounds and equidistribution statements for areal Weil heights.

\begin{Def} \label{lambda heights def}
    Let $K$ be a number field, let $S \subset M_K$ be a finite set of places, and let $\mathbf{t} = (r_v)_{v \in S} \in (0,\infty)^S$ be a tuple of radii indexed by $S$. We define the adelic measure $\lambda_\mathbf{t} = (\lambda_{\mathbf{t}, v})_{v \in M_K}$ by
    \begin{gather*}
        d\lambda_{\mathbf{t},v} = \begin{cases}
            \displaystyle \left. \frac{d\theta}{2\pi r_v} \right |_{\partial D_v (0,r_v)} & \text{if } v \in S \text{ and } v\mid \infty, \\
            d\delta_{\zeta_{0,r_v}} & \text{ if } v \in S \text{ and } v\nmid \infty, \\
            \lambda_v & \text{if } v \notin S.
        \end{cases}
    \end{gather*}
\end{Def}
The adelic measure $\lambda_\mathbf{t}$ should be thought of as the adelic measure associated to $\E_\mathbf{t}$ from (\ref{E_r}). We now prove various properties for the heights $h_{\lambda_\mathbf{t}}$, including an explicit expression for them.

\begin{prop} \label{expression for lambda heights}
    Let $\alpha \in L$, where $L/K$ is a finite extension. For every $w \in M_L$, let $n_w$ be the local-to-global degree of $L/K$ at $w$. For every $w \in M_L^S$, let $t_w = t_v$, where $w \mid v$ and $v \in S$. Then 
    \begin{gather} \label{lambda height expression}
        h_{\lambda_\mathbf{t}} (\alpha) = h_{\lambda_\mathbf{t}} (\infty) + \sum_{w \in M_L} n_w \log^+ \abs{\alpha}_w + \sum_{w \in M_L^S} n_w \log \max \{ t_w, \abs{\alpha}_w \},
    \end{gather}
    where 
    \begin{gather*} 
        h_{\lambda_\mathbf{t}} (\infty) = - \frac{1}{2} \sum_{v \in S} \log t_v^{d_v}.
    \end{gather*}
\end{prop}

\begin{proof}
    The proof follows the same stratey as Theorem \ref{areal height formula}. Note that we have computed that $$p_{\lambda_{\mathbf{t},v}} (z) = \log \max \{ t_v , \abs{z}_v \}$$ in (\ref{potential of lambda v infinite}) and (\ref{potential of lambda v finite}). Thus $$(\lambda_{\mathbf{t}, v}, \lambda_{\mathbf{t}, v}) = \int -\log \max \{ t_v , \abs{z}_v \} \, d\lambda_{\mathbf{t},v} (z) = -\log t_v,$$ which yields $$h_{\lambda_\mathbf{t}} (\infty) = - \frac{1}{2} \sum_{v \in S} \log t_v^{d_v}.$$ Using a computation similar to (\ref{local height}), where $f_{r_v} (\abs{\cdot}_v)$ is replaced with $\log \max \{ t_v , \abs{\cdot}_v \}$, we obtain (\ref{lambda height expression}).
\end{proof}

We now restrict ourselves to studying $h_{\lambda_\mathbf{t}}$ in the special case where $h_{\lambda_\mathbf{t}} (\infty ) = 0$. This is equivalent to the condition that the adelic capacity $\gamma_\infty (\E_\mathbf{t}) = \gamma (\mathbf{t}) = 1$, where $\E_\mathbf{t}$ is defined by (\ref{E_r}). We prove an analog of Kronecker's theorem in this setting. 

\begin{prop} \label{kronecker for lambda heights}
    Let $K$ be a number field, let $S \subset M_K$ be a finite set of places of $K$, and let $\mathbf{t} = (t_v)_{v \in S} \in (0,\infty)^S$ such that $\gamma (\mathbf{t}) = 1$. Suppose that $\alpha \in L$, where $L/K$ is a finite extension of $K$. For every $w \in M_L$, let $n_w$ denote the local-to-global degree of $L/K$ at $w$. For every $w \in M_L^S$, let $t_w = t_v$, where $w \mid v$ and $v \in S$. Then if $\alpha \in \P^1 (\overline{K})$, $h_{\lambda_\mathbf{t}} (\alpha) \geq 0$, with equality if and only if $\alpha = 0$, $\alpha = \infty$, or if the following two conditions hold:
    \begin{enumerate}[\indent (i)]
        \item If $w \in M_L^S$ then $\abs{\alpha}_w = t_w$.
        \item If $w \in M_L \setminus M_L^S$ then $\abs{\alpha}_w = 1$. 
    \end{enumerate}
    Furthermore, $\mathcal{L} (\lambda_\mathbf{t}) = 0$. 
\end{prop}

\begin{proof}
    First observe that by a direct computation $h_{\lambda_\mathbf{t}} (0) = 0 = h_{\lambda_\mathbf{t}} (\infty)$. So we only need to consider the case when $\alpha \in \overline{K}^\times$. Observe that for any $t > 0$ and for all $x \in (0,\infty)$, $\log \max \{ t, x\} \geq \log x$, with equality if and only if $x \geq t$. By this observation and the product formula,
    \begin{gather*}
        h_{\lambda_\mathbf{t}} (\alpha) \geq \sum_{w \in M_L} n_w \log \abs{\alpha}_w = 0,
    \end{gather*}
    where equality holds if and only if the following two conditions hold:
    \begin{enumerate}[\indent (i)]
        \item If $w \in M_L^S$ then $\abs{\alpha}_w \geq t_w$. \label{lambda height nonstrict conditions i}
        \item If $w \in M_L \setminus M_L^S$ then $\abs{\alpha}_w \geq 1$. \label{lambda height nonstrict conditions ii}
    \end{enumerate}
    However, by the product formula and by the fact that $$\sum_{v \in S} d_v \log t_v = 0,$$ conditions (\ref{lambda height nonstrict conditions i}) and (\ref{lambda height nonstrict conditions ii}) can hold if and only if equality holds in both statements.

    Now, we just need to show that $\mathcal{L} (\lambda_\mathbf{t}) = 0$. We use the adelic Fekete-Szeg\"o theorem \cite[Theorem 6.27]{baker-rumely.book}. Let $\varepsilon > 0$ and consider the adelic set
    \begin{gather*}
        \U = \prod_{v \in M_K} U_v,
    \end{gather*}
    where $U_v \subset \P_v^1$ is defined by
    \begin{gather*}
        U_v = \begin{cases}
            \mathcal{D}_v (0,t_v + \varepsilon)^- & \text{if } v \in S,\\
            \mathcal{D}_v (0,1) & \text{if } v \in M_K \setminus S.
        \end{cases}
    \end{gather*}
     Note that $\U$ is an adelic neighborhood of the $K$-symmetric compact Berkovich adelic set $$\E = \prod_{v \in M_K} E_v,$$ where $$E_v = \begin{cases}
         \mathcal{D}_v (0,r_v + \nicefrac{\varepsilon}{2}) & \text{if } v \in S,\\
         \mathcal{D}_v (0,1) & \text{if } v \in M_K \setminus S.
     \end{cases}$$ Furthermore, $$\gamma_\infty (\E) = \prod_{v \in S} (t_v + \nicefrac{\varepsilon}{2})^{d_v} > \prod_{v \in S} t_v^{d_v} = 1.$$ Thus we may use the adelic Fekete-Szeg\"o theorem to find infinitely many $\alpha \in \overline{K}$ such that all Galois conjugates of $\alpha$ are in $\U$ for every $v \in M_K$. Explicitly, this means that the following two conditions hold for any such $\alpha$:
    \begin{enumerate}[\indent (i)]
        \item If $w \in M_{K(\alpha)}^S$ then $\abs{\alpha}_w < t_w + \varepsilon$.
        \item If $w \in M_{K(\alpha)} \setminus M_{K(\alpha)}^S$ then $\abs{\alpha}_w \leq 1$.
    \end{enumerate}
    Thus we have the estimate
    \begin{gather*}
        h_{\lambda_\mathbf{t}} (\alpha) \leq \sum_{v \in S} d_v \log (t_v + \varepsilon).
    \end{gather*}
    By letting $\varepsilon \to 0$, we see that $\mathcal{L} (\lambda_\mathbf{t}) = 0$. 
\end{proof}

For our application of the heights $h_{\lambda_\mathbf{t}}$ to areal Weil heights, we will use the Arakelov-Zhang pairing of adelic measures. In particular, we need the following result.

\begin{prop} \label{pairing of lambda and rho heights}
    Let $S \subset M_K$ be a finite set of places of a number field $K$, and let $\mathbf{r} = (r_v)_{v \in S} \in (0,\infty)^S$. Let $S' \subset M_K$ be a finite set of places such that $S \subset S'$ and let $\mathbf{t} = (t_v)_{v \in S'} \in (0,\infty)^{S'}$. Furthermore, suppose that $\gamma (\mathbf{t}) = 1$. Then we have
    \begin{gather*}
        \AZ{\rho_\mathbf{r}}{\lambda_\mathbf{t}} = h_{\rho_\mathbf{r}} (\infty) + \sum_{v \in S} d_v f_{r_v} (t_v) + \sum_{v \in S'\setminus S} d_v \log^+ t_v.
    \end{gather*}
\end{prop}

\begin{proof}
    Using \cite[Proposition 7]{fili_ui}, we compute 
    \begin{align*}
        \AZ{\rho_\mathbf{r}}{\lambda_\mathbf{t}} = & h_{\rho_\mathbf{r}} (\infty) + h_{\lambda_\mathbf{t}} (\infty) - \sum_{v \in M_K} d_v (\rho_{\mathbf{r},v}, \lambda_{\mathbf{t},v}) \\
        = & h_{\rho_\mathbf{r}} (\infty) + \sum_{v \in S'} d_v \int \log \abs{z-w} \, d\lambda_{\mathbf{t},v} (w) d\rho_{\mathbf{r},v} (z) \\
        = & h_{\rho_\mathbf{r}} (\infty) + \sum_{v \in S} d_v \int f_{r_v} (\abs{z}_v) \, d\lambda_{\mathbf{t},v} (z) + \sum_{v \in S'\setminus S} d_v \int \log^+ \abs{z}_v d\lambda_{\mathbf{t},v} (z) \\
        = & h_{\rho_\mathbf{r}} (\infty) + \sum_{v \in S} d_v f_{r_v} (t_v) + \sum_{v \in S'\setminus S} d_v \log^+ t_v. \qedhere
    \end{align*}
\end{proof}

\begin{rmk}
    An alternative proof of Proposition \ref{pairing of lambda and rho heights} involves finding a sequence of distinct points $(\alpha_n)_{n=1}^\infty \in \P^1 (\overline{K})^\N$ such that $h_{\lambda_\mathbf{t}} (\alpha_n)\to 0$ as $n \to \infty$ and computing $\lim_{n\to\infty} h_{\rho_\mathbf{r}} (\alpha_n)$, which is equal to $\AZ{\rho_\mathbf{r}}{\lambda_\mathbf{t}}$ by Theorem \ref{pst theorem}.
\end{rmk}

\subsection{Analogs of Kronecker's Theorem}

We are now ready to prove Theorems \ref{kronecker when gamma (r) < 1} and \ref{kronecker when gamma (r) > 1}.

\begin{proof}[Proof of Theorem \ref{kronecker when gamma (r) < 1}]
    We compute
    \begin{align}
        h_{\rho_\mathbf{r}} (\alpha) - h_{\rho_\mathbf{r}} (\infty) = &  \; \sum_{w \in M_L} n_w \log^+ \abs{\alpha}_w + \sum_{w \in M_L} n_w f_{r_w} (\abs{\alpha}_w) \nonumber \\
        \geq & \; \sum_{w \in M_L} n_w \log^+ \abs{\alpha}_w + \sum_{w \in M_L} n_w \log \abs{\alpha}_w \label{lower bound initial case computation line 2} \\ \geq & \;
        \sum_{w \in M_L} n_w \log \abs{\alpha}_w \label{lower bound initial case computation line 3} \\
        = & \; 0. \label{lower bound initial case computation line 4}
    \end{align}
    Inequality (\ref{lower bound initial case computation line 2}) follows from Lemma \ref{f_r properties} (\ref{f_r properties a}), and (\ref{lower bound initial case computation line 4}) follows from the product formula. By Lemma \ref{f_r properties} (\ref{f_r properties a}), equality in (\ref{lower bound initial case computation line 2}) holds if and only if $\abs{\alpha}_w \geq r_w$ whenever $w \in M_L^S$. Equality in (\ref{lower bound initial case computation line 3}) holds if and only if $\abs{\alpha}_w \geq 1$ whenever $w \in M_L \setminus M_L^S$.

    Now, suppose that $\gamma (\mathbf{r}) \leq 1$. Pick a finite set of places $S' \subset M_K$ such that $S \subset S'$ and $\mathbf{t} = (t_v)_{v \in S'} \in (0,\infty)^{S'}$ such that $t_v \geq r_v$ for every $v \in S$, $t_v \geq 1$ for every $v \in S'\setminus S$, and $\gamma (\mathbf{t}) = 1$. Using Proposition \ref{kronecker for lambda heights}, find a sequence of distinct algebraic numbers $(\alpha_n)_{n=1}^\infty \in \P^1 (\overline{K})^\N$ such that $h_{\lambda_\mathbf{t}} (\alpha_n) \to 0$ as $n \to \infty$. Then by Theorem \ref{pst theorem},
    \begin{align*}
        \lim_{n\to\infty} h_{\rho_\mathbf{r}} (\alpha_n) = \AZ{\rho_\mathbf{r}}{\lambda_\mathbf{t}}. 
    \end{align*}
    Using Proposition \ref{pairing of lambda and rho heights}, we compute
    \begin{align*}
        \AZ{\rho_\mathbf{r}}{\lambda_\mathbf{t}} = & h_{\rho_\mathbf{r}} (\infty) + \sum_{v \in S} d_v f_{r_v} (t_v) +\sum_{v \in S' \setminus S} d_v \log^+ t_v \\
        = & h_{\rho_\mathbf{r}} (\infty) + \sum_{v \in S'} d_v \log t_v \\
        = & h_{\rho_\mathbf{r}} (\infty).
    \end{align*}
    Thus we have demonstrated that $\mathcal{L}(\rho_\mathbf{r}) = h_{\rho_\mathbf{r}} (\infty)$.
\end{proof}

Suppose that $\alpha \in L^\times$, where $L/K$ is a finite extension, and suppose that $\rho_\mathbf{r}$ is an areal adelic measure with $\gamma (\mathbf{r}) \leq 1$. If (\ref{lower bound initial case i}) and (\ref{lower bound initial case ii}) hold in Theorem \ref{kronecker when gamma (r) < 1}, then we have 
\begin{gather*}
    0 = \sum_{w \in M_L} n_w \log \abs{\alpha}_w \geq \sum_{v \in S} d_v \log r_v = \log \gamma (\mathbf{r}).
\end{gather*}
Thus we cannot have $h_{\rho_\mathbf{r}} (\alpha) = h_{\rho_\mathbf{r}} (\infty)$ if $\gamma (\mathbf{r}) > 1$ (however we could have $h_{\rho_\mathbf{r}} (0) = h_{\rho_\mathbf{r}} (\infty)$). Even if $\gamma (\mathbf{r}) \leq 1$, it might be the case that there does not exist $\alpha \in \overline{K}^\times$ such that $h_{\rho_\mathbf{r}} (\alpha) = h_{\rho_\mathbf{r}} (\infty)$, as the following example demonstrates.

\begin{example}
    Suppose that $r_v$ is transcendental for some $v \in S$ and suppose that $\gamma (\mathbf{r}) = 1$. Then for any $\alpha \in \overline{K}$, we have $\abs{\alpha}_w \neq r_v$ for any $w \in M_{K(\alpha)}$ such that $w \mid v$. Also, since $\gamma (\mathbf{r}) = 1$, by the product formula equality must hold in both (\ref{lower bound initial case i}) and (\ref{lower bound initial case ii}). Thus by Theorem \ref{kronecker when gamma (r) < 1} $h_{\rho_\mathbf{r}} (\alpha) \neq h_{\rho_\mathbf{r}} (\infty)$.
\end{example}

\begin{proof}[Proof of Theorem \ref{kronecker when gamma (r) > 1}]
    Suppose that $\alpha \in L^\times$, where $L/K$ is a finite extension. Following previous notation, for every $w \in M_L^S$, let $t_w = t_v$. Consider the inequality 
    \begin{gather} \label{ineq1 for kronecker when gamma(r) >1}
        c^{-2} (h_{\rho_\mathbf{r}} (\alpha) - \AZ{\rho_\mathbf{r}}{\lambda_\mathbf{t}}) \geq h_{\rho_\mathbf{t}} (\alpha) - h_{\rho_\mathbf{t}} (\infty).
    \end{gather}
    A direct computation gives that the right-hand side of (\ref{ineq1 for kronecker when gamma(r) >1}) is
    \begin{multline} \label{ineq2 for kronecker when gamma(r) >1}
        c^{-2} (h_{\rho_\mathbf{r}} (\alpha) - \AZ{\rho_\mathbf{r}}{\lambda_\mathbf{t}} ) = \sum_{w \in M_L \setminus M_L^S} n_w c^{-2} \log^+ \abs{\alpha}_w \\+ \sum_{w \in M_L^S} n_w c^{-2} (f_{r_w} (\abs{\alpha}_w) - f_{r_w} (t_w) ).
    \end{multline}
    Since $c^{-2} > 1$, we have that 
    \begin{multline} \label{ineq3 for kronecker when gamma(r) >1}
        c^{-2} (h_{\rho_\mathbf{r}} (\alpha) - \AZ{\rho_\mathbf{r}}{\lambda_\mathbf{t}}) \geq 
        \sum_{w \in M_L \setminus M_L^S} n_w \log^+ \abs{\alpha}_w \\+ \sum_{w\in M_L^S} n_w c^{-2} (f_{r_w} (\abs{\alpha}_w) - f_{r_w} (t_w) ),
    \end{multline}
    with equality if and only if $\abs{\alpha}_w \leq 1$ whenever $w \in M_L \setminus M_L^S$. Furthermore, by elementary calculus, 
    \begin{gather*} 
        c^{-2} (f_{r_w} (x) - f_{r_w} (t_w)) \geq f_{t_w} (x) - \log t_w,
    \end{gather*}
    with equality if and only if $x \leq t_w$. This implies that 
    \begin{align}
    \begin{split} \label{ineq4 for kronecker when gamma(r) >1}
        \sum_{w \in M_L^S} n_w c^{-2} (f_{r_w} (\abs{\alpha}_w) - f_{r_w} (t_w) ) \geq & \sum_{w \in M_L^S} n_w (f_{t_w} (\abs{\alpha}_w) - \log t_w) \\ = & \sum_{w \in M_L^S} n_w f_{t_w} (\abs{\alpha}_w),
    \end{split}
    \end{align}
    where equality holds if and only if $\abs{\alpha}_w \leq t_w$ whenever $w \in M_L^S$. Combining (\ref{ineq2 for kronecker when gamma(r) >1}), (\ref{ineq3 for kronecker when gamma(r) >1}), and (\ref{ineq4 for kronecker when gamma(r) >1}), we obtain
    \begin{align*}
        \begin{split}
            c^{-2} (h_{\rho_\mathbf{r}} (\alpha) - \AZ{\rho_\mathbf{r}}{\lambda_\mathbf{t}}) \geq & \sum_{w \in M_L \setminus M_L^S} n_w \log^+ \abs{\alpha}_w + \sum_{w \in M_L^S} n_w f_{t_w} (\abs{\alpha}_w) \\
            = & h_{\rho_\mathbf{t}} (\alpha) - h_{\rho_\mathbf{t}} (\infty),
        \end{split}
    \end{align*}
    with equality if and only if the following two conditions hold:
    \begin{enumerate}[\indent (i)]
        \item If $w \in M_L^S$ then $\abs{\alpha}_w \leq t_w$.
        \item If $w \in M_L \setminus M_L^S$ then $\abs{\alpha}_w \leq 1$.
    \end{enumerate}
    Since $\gamma (\mathbf{t}) = 1$, by the product formula, the above conditions can only hold when both inequalities are replaced by equalities. Furthermore, these are precisely the conditions for $h_{\rho_\mathbf{t}} (\alpha) = h_{\rho_\mathbf{t}} (\infty)$ by Theorem \ref{kronecker when gamma (r) < 1}.

    Now we just need to show that $\mathcal{L}(\rho_\mathbf{r}) = \AZ{\rho_\mathbf{r}}{\lambda_\mathbf{t}}$. Using Proposition \ref{kronecker for lambda heights}, find a sequence of distinct points $(\alpha_n)_{n=1}^\infty \in \P^1 (\overline{K})^\N$ such that $h_{\lambda_\mathbf{t}} (\alpha_n) \to 0$ as $n \to \infty$. Then by Theorem \ref{pst theorem},
    \begin{gather*}
        \lim_{n\to\infty} h_{\rho_\mathbf{r}} (\alpha_n) = \AZ{\rho_\mathbf{r}}{\lambda_\mathbf{t}}.    \qedhere
    \end{gather*}    
\end{proof}

We now summarize what we have shown about the spectrum of values $h_{\rho_\mathbf{r}} (\P^1 (\overline{K}))$ for an areal Weil height $h_{\rho_\mathbf{r}}$. The following is a corollary of Theorems \ref{kronecker when gamma (r) < 1} and \ref{kronecker when gamma (r) > 1}.

\begin{cor} \label{spectrum of areal height corollary}
    If $\rho_\mathbf{r}$ is an areal adelic measure, then $\mathcal{L} (\rho_\mathbf{r}) > 0$. Furthermore, $h_{\rho_\mathbf{r}} (0)$ is an isolated point of $h_{\rho_\mathbf{r}} (\P^1 (\overline{K}))$. When $\gamma (\mathbf{r}) > 1$, then $h_{\rho_\mathbf{r}} (\infty)$ is also an isolated point of $h_{\rho_\mathbf{r}} (\P^1 (\overline{K}))$ . 
\end{cor}

\begin{proof}
    Given $\rho_\mathbf{r}$, we may find $\lambda_\mathbf{t}$ such that $\mathcal{L} (\rho_\mathbf{r}) = \AZ{\rho_\mathbf{r}}{\lambda_\mathbf{t}}$. This was done in the proof of Theorem \ref{kronecker when gamma (r) < 1} when $\gamma (\mathbf{r}) \leq 1$ and in the proof of Theorem \ref{kronecker when gamma (r) > 1} when $\gamma (\mathbf{r}) > 1$. By \cite[Theorem 1]{fili_ui} we have $\AZ{\rho_\mathbf{r}}{\lambda_\mathbf{t}} > 0$.

    By the definition of $\mathcal{L}(\rho_\mathbf{r})$, if $\alpha \in \P^1 (\overline{K})$ such that $h_{\rho_\mathbf{r}} (\alpha) < \mathcal{L}(\rho_\mathbf{r})$, then $h_{\rho_\mathbf{r}} (\alpha)$ is isolated in $h_{\rho_\mathbf{r}} (\P^1 (\overline{K}))$. In the case where $\gamma (\mathbf{r}) \leq 1$, we have 
    \begin{gather*}
        h_{\rho_\mathbf{r}} (0) = h_{\rho_\mathbf{r}} (\infty) +  \sum_{v \in S} d_v \left ( \log r_v - \frac{1}{2} \right ) 
        \leq  h_{\rho_\mathbf{r}} (\infty) + \sum_{v \in S} -\frac{d_v}{2} < h_{\rho_\mathbf{r}} (\infty) = \mathcal{L} (\rho_\mathbf{r}),
    \end{gather*}
    demonstrating that $h_{\rho_\mathbf{r}}(0)$ is isolated in $h_{\rho_\mathbf{r}} (\P^1 (\overline{K}))$.
    If $\gamma (\mathbf{r}) > 1$, then by Proposition \ref{pairing of lambda and rho heights} and Theorem \ref{kronecker when gamma (r) > 1} we have 
    $$\mathcal{L} (\rho_\mathbf{r}) = \AZ{\rho_\mathbf{r}}{\lambda_\mathbf{t}} = h_{\rho_\mathbf{r}} (\infty) + \sum_{v \in S} d_v \left ( \log r_v - \frac{1}{2} + \frac{t_v^2}{2 r_v^2} \right ),$$ where $\mathbf{t} = (t_v)_{v \in S} = (cr_v)_{v \in S}$ is defined as in Theorem \ref{kronecker when gamma (r) > 1}. Since $0 < \nicefrac{t_v^2}{2r_v^2}$ for any $v \in S$, we get that $h_{\rho_\mathbf{r}} (0) < \mathcal{L} (\rho_\mathbf{r})$. Also,
    \begin{align*}
        \sum_{v \in S} d_v \left ( \log r_v - \frac{1}{2} + \frac{t_v^2}{2 r_v^2} \right ) = & \sum_{v \in S} d_v \left ( \log c^{-1} + \log t_v - \frac{1}{2} + \frac{1}{2c^{-2}} \right ) \\
        = & \sum_{v \in S} d_v \left ( \log c^{-1} - \frac{1}{2} + \frac{1}{2c^{-2}} \right ) \\
        > & 0,
    \end{align*}
    which implies that $h_{\rho_\mathbf{r}} (\infty) < \mathcal{L} (\rho_\mathbf{r})$. Thus we have shown that when $\gamma (\mathbf{r}) > 1$, both $h_{\rho_\mathbf{r}} (0)$ and $h_{\rho_\mathbf{r}} (\infty)$ are isolated in $h_{\rho_\mathbf{r}} (\P^1 (\overline{K}))$. 
\end{proof}

\section{Equidistribution} \label{equidistribution section}

The goal of this section is to prove Proposition \ref{equidistribution for gamma (r) < 1} and Theorem \ref{equidistribution for gamma (r) > 1}. We begin with Proposition \ref{equidistribution for gamma (r) < 1}, which follows from the proof of Theorem \ref{kronecker when gamma (r) < 1}. 

\begin{proof}[Proof of Proposition \ref{equidistribution for gamma (r) < 1}]
    In the proof of Theorem \ref{kronecker when gamma (r) < 1}, we computed that $\mathcal{L} (\rho_\mathbf{r}) = \AZ{\rho_\mathbf{r}}{\lambda_\mathbf{t}} = h_{\rho_\mathbf{r}} (\infty)$. Thus by Theorem \ref{pst theorem}, for any sequence of distinct points $(\alpha_n)_{n=1}^\infty \in \P^1 (\overline{K})^\N$ such that $h_{\lambda_\mathbf{t}} (\alpha_n) \to 0$, we have $h_{\rho_\mathbf{r}} (\alpha_n) \to h_{\rho_\mathbf{r}} (\infty)$. Furthermore, such a sequence exists by Proposition \ref{kronecker for lambda heights}. For any such sequence, by Theorem \ref{frl thm 2}, $[\alpha_n]_v \overset{*}{\to} \lambda_{\mathbf{t}, v}$ for all $v \in M_K$. 
\end{proof}

Next, we prove Theorem \ref{equidistribution for gamma (r) > 1}, which is an analog of Bilu's theorem for areal Weil heights $h_{\rho_\mathbf{r}}$ with $\gamma (\rho_\mathbf{r}) \geq 1$. For our proof, we follow the convention for writing $h_{\rho_\mathbf{r}}$ that is given in Remark \ref{another expression for areal weil heights}.

\begin{proof}
    First suppose that $h_{\lambda_\bt} (\alpha_n) \to 0$ as $n \to \infty$. Then Theorem \ref{pst theorem} implies that $h_{\rho_\mathbf{r}} (\alpha_n) \to \AZ{\rho_\mathbf{r}}{\lambda_\bt}$ as $n \to \infty$.

    Now suppose that $h_{\rho_\mathbf{r}} (\alpha_n) \to \AZ{\rho_\mathbf{r}}{\lambda_\bt}$ as $n \to \infty$. We first consider the case where $c=1$ and $\br = \bt$. In this case, $\AZ{\rho_\mathbf{r}}{\lambda_\mathbf{t}} = h_{\rho_\mathbf{r}} (\infty)$ by Theorem \ref{kronecker when gamma (r) < 1}. For each $n \in \N$, let $\ell_n = [K(\alpha_n) : K]$ and let $\alpha_{n,1}, \cdots, \alpha_{n,\ell_n} \in \overline{K}$ be the distinct Galois conjugates of $\alpha_n$ over $K$. We have
    \begin{multline} \label{equidistribution equation 1}
        h_{\rho_\mathbf{r}} (\alpha_n) - h_{\rho_\mathbf{r}} (\infty) =  \sum_{v \in M_K \setminus S} \frac{d_v}{\ell_n} \sum_{k=1}^{\ell_n} \log^+ \abs{\alpha_{n,k}}_v \\ + \sum_{v \in S} \frac{d_v}{\ell_n} \left ( \sum_{\abs{\alpha_{n,k}}_v < r_v} \left ( \log r_v - \frac{1}{2} + \frac{\abs{\alpha_{n,k}}_v^2}{2r_v^2} \right ) + \sum_{\abs{\alpha_{n,k}}_v \geq r_v} \log \abs{\alpha_{n,k}}_v  \right ).
    \end{multline}
    Note by the product formula that 
    \begin{multline*}
        \sum_{v \in M_K \setminus S} \frac{d_v}{\ell_n} \sum_{k=1}^n \log^+ \abs{\alpha_{n,k}}_v + \sum_{v \in S} \frac{d_v}{\ell_n} \sum_{\abs{\alpha_{n,k}}_v \geq r_v} \log \abs{\alpha_{n,k}}_v \\ = -\sum_{v \in M_K \setminus S} \frac{d_v}{\ell_n} \sum_{\abs{\alpha_{n,k}}_v < 1} \log \abs{\alpha_{n,k}}_v - \sum_{v \in S} \frac{d_v}{\ell_n} \sum_{\abs{\alpha_{n,k}}_v < r_v} \log \abs{\alpha_{n,k}}_v.
    \end{multline*}
    Thus (\ref{equidistribution equation 1}) becomes
    \begin{multline*} 
        h_{\rho_\mathbf{r}} (\alpha_n) - h_{\rho_\mathbf{r}} (\infty) = \sum_{v \in M_K \setminus S} \frac{d_v}{\ell_n} \sum_{\abs{\alpha_{n,k}}_v < 1} -\log \abs{\alpha_{n,k}}_v \\+ \sum_{v \in S} \frac{d_v}{\ell_n} \sum_{\abs{\alpha_{n,k}}_v < r_v} \left (\log r_v - \frac{1}{2} + \frac{\abs{\alpha_{n,k}}_v^2}{2r_v^2} - \log \abs{\alpha_{n,k}}_v \right ).
    \end{multline*}
    This implies that 
    \begin{gather*}
        h_{\rho_\mathbf{r}} (\alpha_n) - h_{\rho_\mathbf{r}} (\infty) \geq \sum_{v \in S} \frac{d_v}{\ell_n} \sum_{\abs{\alpha_{n,k}}_v < r_v} \left (\log r_v - \frac{1}{2} + \frac{\abs{\alpha_{n,k}}_v^2}{2r_v^2} - \log \abs{\alpha_{n,k}}_v \right ).
    \end{gather*}
    By Lemma \ref{f_r properties} (\ref{f_r properties a}) we know that $$\log r_v - \frac{1}{2} + \frac{\abs{\alpha_{n,j}}_v^2}{2r_v^2} - \log \abs{\alpha_{n,j}} \geq 0$$ and that $$\log r_v - \frac{1}{2} + \frac{x^2}{2r_v^2} - \log x$$ is strictly decreasing for $x \in (0,r_v)$.
    
    Now fix $v \in S$. Following the proof strategy of Theorem 2.1 in \cite{pritsker-areal}, we claim that $o(\ell_n)$ conjugates of $\alpha_n$ are in $\mathcal{D}_v (0,a)$ for any $a<r_v$. For a contradiction, assume that there exists $c > 0$ such that for any $N \in \N$ we may find $n \geq N$ where $\alpha_{n}$ has at least $c \ell_n$ conjugates in $\mathcal{D}_v (0,a)$. Then for such a choice of $n$, we have
    \begin{align*}
        h_{\rho_\mathbf{r}} (\alpha_n) - h_{\rho_\mathbf{r}} (\infty) \geq & \frac{d_v}{\ell_n} \sum_{\abs{\alpha_{n,k}}_v < r_v} \left (\log r_v - \frac{1}{2} + \frac{\abs{\alpha_{n,k}}_v^2}{2r_v^2} - \log \abs{\alpha_{n,k}}_v \right ) \\
        \geq &  d_v c \left ( \log r_v - \frac{1}{2} + \frac{a^2}{2r_v^2} - \log a \right ) \\ > & 0.
    \end{align*}
    Since the choice of $n$ in the above inequality can be made to be arbitrarily large, this contradicts the assumption that $h_{\rho_\mathbf{r}} (\alpha_n) \to \AZ{\rho_\mathbf{r}}{\lambda_\bt}$ as $n \to \infty$. 

    So $m_n = o(\ell_n)$ conjugates of $\alpha_n$ are in $\mathcal{D}_v (0, a)$ for any $a < r_v$. This implies that 
    \begin{align*}
        \lim_{n\to\infty} \left (\frac{d_v}{\ell_n} \sum_{\abs{\alpha_{n,k}}_v < r_v} -\frac{1}{2} + \frac{\abs{\alpha_{n,k}}_v^2}{2r_v^2} \right ) \geq & \lim_{n\to\infty} \left ( \frac{d_v}{\ell_n} \left ( - \frac{m_n}{2} + \ell_n \left ( - \frac{1}{2} + \frac{a^2}{2r_v^2} \right ) \right ) \right ) \\
        \geq & d_v \left ( -\frac{1}{2} + \frac{a^2}{2r_v^2} \right ).
    \end{align*}
    Letting $a \to r_v$, we see that 
    \begin{gather} \label{equidistribution equation 3}
        \lim_{n\to\infty} \left (\frac{d_v}{\ell_n} \sum_{\abs{\alpha_{n,k}}_v < r_v} -\frac{1}{2} + \frac{\abs{\alpha_{n,k}}_v^2}{2r_v^2} \right ) \geq 0.
    \end{gather}
    But also $$\frac{d_v}{\ell_n} \sum_{\abs{\alpha_{n,k}}_v < r_v} -\frac{1}{2} + \frac{\abs{\alpha_{n,k}}_v^2}{2r_v^2} \leq 0$$ for every $n$, so equality must hold in (\ref{equidistribution equation 3}). 

    We have shown that for any $v \in S$ we must have 
    \begin{gather} \label{equidistribution equation 4}
        \lim_{n\to\infty} \left ( \frac{d_v}{\ell_n} \sum_{\abs{\alpha_{n,k}}_v < r_v} -\frac{1}{2} + \frac{\abs{\alpha_{n,k}}_v^2}{2r_v^2} \right ) = 0. 
    \end{gather}
    Thus using (\ref{equidistribution equation 1}) and (\ref{equidistribution equation 4}) we have
    \begin{align*}
        0 = & \lim_{n\to\infty} ( h_{\rho_\mathbf{r}} (\alpha_n) - h_{\rho_\mathbf{r}} (\infty)) \\ = & \lim_{n\to\infty} \left ( \sum_{v \in M_K \setminus S} \frac{d_v}{\ell_n} \sum_{k=1}^{\ell_n} \log^+ \abs{\alpha_{n,k}}_v + \sum_{v \in S} \frac{d_v}{\ell_n} \sum_{k=1}^{\ell_n} \log \max \{ r_v, \abs{\alpha_{n,k}}_v \} \right ) \\
        = & \lim_{n\to\infty} h_{\lambda_\bt} (\alpha_n). 
    \end{align*}
    Therefore the result holds for $c=1$. 

    Now suppose that $c > 1$. Using Theorem \ref{kronecker when gamma (r) > 1} we have 
    \begin{gather*}
        0 =  \lim_{n\to\infty} (h_{\rho_\mathbf{r}} (\alpha_n ) - \AZ{\rho_\mathbf{r}}{\lambda_\bt}) \geq  \lim_{n\to\infty} c^2 (h_{\rho_\bt} (\alpha_n) - h_{\rho_\mathbf{t}} (\infty)) \geq 0.
    \end{gather*}
    This implies that $\lim_{n\to\infty} h_{\rho_\bt} (\alpha_n) = h_{\rho_\mathbf{t}} (\infty)$. Thus by the result for $c=1$ we have that $h_{\lambda_\bt} (\alpha_n) \to 0$. 

    Finally, whenever $h_{\lambda_\bt} (\alpha_n) \to 0$, by Theorem \ref{frl thm 2} we have that for every $v \in M_K$, $[\alpha_n]_v \overset{*}{\to} \lambda_{\bt , v}$ as $n \to \infty$.
\end{proof}

As a consequence of Theorem \ref{equidistribution for gamma (r) > 1}, we are now able to prove Corollary \ref{peters observation}.

\begin{proof}[Proof of Corollary \ref{peters observation}]
    Suppose that $\mu$ is an adelic measure such that $\mathcal{L} (\mu) = 0$ and without loss of generality assume that $\mu$ is defined over $K$. Write $\mu = (\mu_v)_{v \in M_K}$. 
    
    If $(\alpha_n)_{n=1}^\infty \in \P^1 (\overline{K})^\N$ is a sequence of distinct algebraic numbers such that $h_\mu (\alpha_n) \to 0$, then by Theorem \ref{pst theorem}
    \begin{gather} \label{unique closest measure ineq}
        \AZ{\rho_\mathbf{r}}{\mu} = \lim_{n\to\infty} h_{\rho_\mathbf{r}} (\alpha_n) \geq \mathcal{L}(\rho_\mathbf{r}) = \AZ{\rho_\mathbf{r}}{\lambda_\bt}.
    \end{gather}
    If equality holds in (\ref{unique closest measure ineq}), then by Theorem \ref{equidistribution for gamma (r) > 1}, $[\alpha_n]_v \overset{*}{\to} \lambda_{\bt , v}$ as $n \to \infty$ for every $v \in M_K$. But also $h_\mu (\alpha_n) \to 0$, so for every $v \in M_K$ we have $[\alpha_n]_v \overset{*}{\to} \mu_v$ as $n \to \infty$. Therefore $\mu_v = \lambda_{\bt , v}$ for every $v \in M_K$, which implies $\mu = \lambda_\bt$. 
\end{proof}

\section{Examples} \label{examples section}

\subsection{Lehmer's Conjecture and Areal Weil Heights} \label{lehmers question for areal weil heights} Let $K$ be a number field, $S\subset M_K$ be a nonempty, finite set of places of $K$, and let $\mathbf{r} = (1)_{v \in S}$. We will consider an analog of Lehmer's Conjecture \ref{lehmers conjecture} for areal Weil heights $h_{\rho_\mathbf{r}}$. We first note that there are different ways that one may generalize Lehmer's Conjecture to an adelic height. For instance, suppose that $\mu$ is an adelic measure defined over a number field $K$. One may conjecture the following. 

\begin{conj} \label{lehmer for adelic v1}
    There exists $\varepsilon > 0$ such that $[K(\alpha) : K] h_\mu (\alpha) \geq \varepsilon$ for every $\alpha \in \P^1 (\overline{K})$ with $h_\mu (\alpha) > 0$.
\end{conj} 

However, Conjecture \ref{lehmer for adelic v1} is trivially true if $\mathcal{L} (\mu) > 0$. Thus this is not the right generalization to consider for $\mu = \rho_\mathbf{r}$. We instead consider the following analog of Lehmer's Conjecture, which accounts for the case where $\mathcal{L} (\mu) > 0$.

\begin{conj} \label{lehmer for adelic v2}
    There exists $\varepsilon > 0$ such that $$[K(\alpha) : K] (h_\mu (\alpha) - \mathcal{L} (\mu )) \geq \varepsilon$$ for every $\alpha \in \P^1 (\overline{K})$ such that $h_\mu (\alpha) > \mathcal{L}(\mu)$.
\end{conj}

We now show that Conjecture \ref{lehmer for adelic v2} is false for $\mu = \rho_\mathbf{r}$, where $\mathbf{r} = (1)_{v \in S}$. Recall that in this case, $\mathcal{L} (\rho_\mathbf{r}) = h_{\rho_\mathbf{r}} (\infty)$ by Theorem \ref{kronecker when gamma (r) < 1}. 

\begin{prop} \label{a sequence fails lehmer}
    Suppose that $\br = (1)_{v \in S}$, where $S$ is a nonempty set of places of a number field $K$. Then there exists a sequence $(\alpha_n)_{n=1}^\infty \in \P^1 (\overline{K})^\N$ of distinct points such that $$\lim_{n\to\infty} [K(\alpha_n) : K](h_{\rho_\mathbf{r}} (\alpha_n ) - h_{\rho_\mathbf{r}} (\infty)) = 0.$$
\end{prop}

\begin{proof}
    Let $S' = \{ v \in M_K : v \in S \text{ or } v \mid \infty \} \cup \{v_0\}$, where $v_0 \in M_K \setminus S$ is finite. Let $U_{K,S'}$ be the $S'$-unit group of $K$, that is $$U_{K,S'} = \{ \alpha \in K : \abs{\alpha}_v = 1 \text{ for all } v \in M_K \setminus S' \}.$$ Dirichlet's $S$-unit theorem implies that we may find $\alpha \in U_{K, S'}$ such that $\abs{\alpha}_v < 1$ for $v \in S$ and $\abs{\alpha}_v \geq 1$ for $v \in S' \setminus S$. Fix such an $\alpha \in K$. 

    Now, to construct the sequence $(\alpha_n)_{n=1}^\infty$, we first note that by the discreteness of the logarithmic embedding of $U_{K,S'}$ into $\R^{S'}$, there exists $\ell \in \N$ such that for all $n \geq \ell$, no root of $x^n - \alpha$ is in $K$. Fix such an $\ell$. Then for all primes $p \geq \ell$, the polynomial $x^p - \alpha$ is irreducible over $K$. So let $\{p_n\}_{n=1}^\infty$ be an increasing sequence of prime numbers with $p_1 \geq \ell$. For each $n \in \N$, let $\alpha_n$ be a root of $x^{p_n} - \alpha$ and let $L_n = K(\alpha_n)$.

    Note that for any $w \in M_{L_n}$ we have $\abs{\alpha_n}_w = \abs{\alpha}_v^{\nicefrac{1}{p_n}}$, where $w \mid v$ and $v \in M_K$. Thus we have $\abs{\alpha_n}_w < 1$ when $w \in M_{L_n}^S$ and$\abs{\alpha_n}_w \geq 1$ when $w \in M_{L_n} \setminus M_{L_n}^S$. We compute
    \begin{align} 
    \begin{split} \label{sequence failing lehmer eq1}
        h_{\rho_\mathbf{r}} (\alpha_n) - h_{\rho_\mathbf{r}} (\infty) = & \sum_{w \in M_{L_n} \setminus M_{L_n}^S} n_w \log^+ \abs{\alpha_n}_w + \sum_{w \in M_{L_n}^S} n_w f_1 (\abs{\alpha_n}_w)\\
        = & \sum_{w \in M_{L_n} \setminus M_{L_n}^S} n_w \log \abs{\alpha_n}_w + \sum_{w \in M_{L_n}^S} n_w \frac{\abs{\alpha_n}_w^2 - 1}{2}.
        \end{split}
    \end{align}
    Using the product formula, we have
    \begin{gather} \label{sequence failing lehmer eq2}
        \sum_{w \in M_{L_n} \setminus M_{L_n}^S} n_w \log \abs{\alpha_n}_w = -\sum_{w \in M_{L_n}^S} n_w \log \abs{\alpha_n}_w.
    \end{gather}
    Combining (\ref{sequence failing lehmer eq1}) and (\ref{sequence failing lehmer eq2}) we obtain
    \begin{align*}
        h_{\rho_\mathbf{r}} (\alpha_n) - h_{\rho_\mathbf{r}} (\infty) = & \sum_{w \in M_{L_n}^S} n_w \left ( \frac{\abs{\alpha_n}_w^2 - 1}{2} - \log \abs{\alpha_n}_w \right ) \\
        = & \sum_{v \in S} d_v \left ( \frac{\abs{\alpha}_v^{\nicefrac{2}{p_n}} - 1}{2} - \log \abs{\alpha}_v^{\nicefrac{1}{p_n}} \right ).
    \end{align*}
    Now, notice that 
    \begin{gather*}
        \abs{\alpha}_v^{\nicefrac{2}{p_n}} = \exp \left ( \frac{2}{p_n} \log \abs{\alpha}_v \right ) = 1 + \frac{2}{p_n} \log \abs{\alpha}_v + O \left ( \frac{1}{p_n^2} \right ).
    \end{gather*}
    Using the above, 
    \begin{align*}
         h_{\rho_\mathbf{r}} (\alpha_n) - h_{\rho_\mathbf{r}} (\infty) = O \left ( \frac{1}{p_n^2} \right ).
    \end{align*}
    Noting that $p_n = [L_n : K]$, we find that $p_n (h_{\rho_\mathbf{r}} (\alpha_n) - h_{\rho_\mathbf{r}} (\infty)) \to 0$ as $n \to \infty$. 
\end{proof}

\begin{rmk}
    An alternative method to construct the sequence $(\alpha_n)_{n=1}^\infty$ is to consider polynomials $x^n - \alpha$ for all $n \in \N$. Observe that $$B = \{\beta \in \overline{K} : \beta^n - \alpha = 0 \text{ for some } n \in \N\}$$ is a set of bounded Weil height, thus has only finitely many elements of bounded degree over $K$. So we may also prove Proposition \ref{a sequence fails lehmer} by letting $(\alpha_n)_{n=1}^\infty$ be an enumeration of $B$.
\end{rmk}

\begin{rmk}
    Let $c > 1$ and consider $\rho = \rho_{c \mathbf{r}}$, where $\mathbf{r} = (1)_{v \in S}$ and where $(\alpha_n)_{n=1}^\infty$ is the sequence from Proposition \ref{a sequence fails lehmer}. By repeating computations similar to those in the proof of Propostion \ref{a sequence fails lehmer}, one may show that $$\lim_{n\to \infty} p_n (h_\rho (\alpha_n) - \mathcal{L}(\rho_\mathbf{r})) = \frac{1-c^2}{c^2} \sum_{v \in S} d_v \log \abs{\alpha}_v = \frac{c^2 - 1}{c^2} h(\alpha) > 0.$$ Thus our method for proving that Conjecture \ref{lehmer for adelic v2} is false when $\mathbf{r} = (1)_{v \in S}$ does not extend to the case with $c\mathbf{r}$, where $c > 1$. We conjecture that Conjecture \ref{lehmer for adelic v2} is false for $\mu = \rho_\mathbf{r}$ with $\gamma (\mathbf{r}) \geq 1$ if and only if $\gamma (\mathbf{r}) = 1$. 
\end{rmk}

\subsection{A Noteworthy Class of Areal Heights} \label{special class of areal}

For the rest of the paper we specialize to the following collection of areal adelic measures. Let $K = \Q$, $S = \{\infty \}$, and let $r_\infty = r >0$. Since $S$ is a singleton, we will abuse notation and write $r$ in place of $\mathbf{r} = (r_\infty)$. We will also write $\abs{\cdot}$ in place of $\abs{\cdot}_\infty$. Our goal is to study $\rho_r$ and its behavior as $r$ varies. Observe that in this setting,
\begin{gather*}
    h_{\rho_r} (\infty) = \frac{1}{8} - \frac{1}{2} \log r.
\end{gather*}
Suppose $\alpha \in \overline{\Q}$ has minimal polynomial $P_\alpha (z) = a_n z^n + \cdots + a_0 \in \Z [z]$ over $\Z$. Then by Proposition 1.3 of \cite{frl.2} we have 
\begin{align*}
    h_{\rho_r} (\alpha ) = & h_{\rho_r} (\infty)  + \frac{1}{n} \sum_{v \in M_\Q} \int_{\sP_v^1} \log \abs{\frac{P(z)}{a_n}}_v d\rho_{r,v} (z) \\
    = & h_{\rho_r} (\infty)  + \frac{1}{n} \sum_{v \in M_K} \int_{\sP_v^1} \log \abs{P(z)}_v d\rho_{r,v} (z) - \frac{1}{n} \sum_{v \in M_K} \log \abs{a_n}_v\\
    = & h_{\rho_r} (\infty) + \frac{1}{n} \int \log \abs{P_\alpha (z)} d\rho_{r,\infty} (z).
\end{align*}
The last line follows from the product formula and the fact that the only nontrivial contribution to the first sum is at $v = \infty$. We observe the special case
$$h_{\rho_1} (\alpha) - h_{\rho_1} (\infty) = \frac{1}{n} m_\D (\alpha),$$ where $m_\D$ denotes the areal Mahler measure.

If $r \leq 1$, then by Theorem \ref{kronecker when gamma (r) < 1} we have 
\begin{gather*}
\mathcal{L}(\rho_r) = \frac{1}{8} - \frac{1}{2}\log r.
\end{gather*}
If $r > 1$, then by Theorem \ref{kronecker when gamma (r) > 1} and Example \ref{height of zeta} we have 
\begin{gather*}
    \mathcal{L} (\rho_r) = \AZ{\rho_r}{\lambda} = \frac{1}{2} \log r - \frac{3}{8} + \frac{1}{2r^2}.
\end{gather*}

\subsection{Uniform Equidistribution in the Complex Plane} \label{uniform equidistribution}

The goal for this section is to determine for which $r$ we have that $\rho_{r,\infty}$ is a limiting distribution for a sequence of complete sets of conjugate algebraic integers. To be precise, we recall the definition of an arithmetic measure, due to Orloski and Sardari.

\begin{Def}[\cite{orloski-sardari}, Definition 1.1]
    Suppose that $\mu$ is a probability measure supported on a compact set $E \subset \C$ which is symmetric about the real line. Then $\mu$ is arithmetic if there exists a sequence of algebraic integers $(\alpha_n)_{n=1}^\infty$ such that the following two conditions hold:
    \begin{enumerate}
        \item For any neighborhood $U \supset E$, there exists $N \in \N$ such that all Galois conjugates of $\alpha_n$ in $\C$ are in $U$ for all $n \geq N$. \label{arithmetic cond 1}
        \item The Galois conjugates of $\alpha_n$ equidistribute to $\mu$, i.e. $[\alpha_n]_\infty \overset{*}{\to} \mu$ as $n\to\infty$. \label{arithmetic cond 2} 
    \end{enumerate}
\end{Def}

We will prove the following theorem.

\begin{thm} \label{when it's arithmetic}
    The measure $\rho_{r,\infty}$ is arithmetic if and only if $r \geq e^{\nicefrac{1}{2}}$.
\end{thm}

\begin{proof}
    By \cite[Theorem 1.2]{orloski-sardari}, it is sufficient to determine for which $r$ we have that for any nonzero integer polynomial $Q(z) \in \Z[z]$, 
    \begin{gather} \label{OS condition}
        \int \log \abs{Q(z)} d\rho_{r,\infty} (z) \geq 0.
    \end{gather}
    Note that if the potential $p_{\rho_{r,\infty}} (z)$ is nonnegative then (\ref{OS condition}) clearly holds. Recall that $p_{\rho_{r,\infty}} (z) = f_r (\abs{z})$, where $f_r$ is given by (\ref{potential function}). By Proposition \ref{f_r properties}, $f_r (x)$ has the global minimum $$f_r (0) = \log r - \frac{1}{2},$$ which is nonnegative precisely when $r \geq e^{\nicefrac{1}{2}}$. On the other hand, if $r < e^{\nicefrac{1}{2}}$, then $$\int \log \abs{z} d\rho_{r,\infty} (z) = p_{\rho_{r,\infty}} (0) < 0.$$ Thus we have that (\ref{OS condition}) holds for all $Q(z) \in \Z [z]$ precisely when $r \geq e^{\nicefrac{1}{2}}$.
\end{proof}

\begin{example} \label{uniform equidistribution example}
Suppose that $r \geq e^{\nicefrac{1}{2}}$ and use Theorem \ref{when it's arithmetic} to find a sequence $(\alpha_n)_{n=1}^\infty$ of algebraic integers satisfying (\ref{arithmetic cond 1}) and (\ref{arithmetic cond 2}). From these conditions, we have 
\begin{gather*}
    \lim_{n\to\infty} \int f(z) d[\alpha_n]_\infty (z) = \int f(z) d_{\rho_{r,\infty}} (z)
\end{gather*}
for any function $f: U \to \R$ which is continuous on a neighborhood of $U$ of $r\D$. 
Note that for any $n$ and for any finite $v$, since $\alpha_n$ is an algebraic integer, $$(\rho_r , [\alpha_n])_v = (\lambda , [\alpha_n])_v = 0.$$ 
So we may compute 
\begin{align*}
    \lim_{n\to \infty} h_{\rho_r} (\alpha_n) = & \lim_{n\to \infty} h_{\rho_r} (\infty) - \sum_{v \in M_K} (\rho_r , [\alpha_n])_v \\
    = & \lim_{n\to\infty} h_{\rho_r} (\infty) + \iint \log \abs{z-w} d\rho_{r,\infty} (z) d[\alpha_n]_\infty (w) \\
    = & \lim_{n\to\infty} h_{\rho_r} (\infty ) + \int p_{\rho_{r,\infty}} (w) d[\alpha_n]_\infty (w) \\
    = & h_{\rho_r} (\infty) - (\rho_r ,\rho_r)_\infty \\
    = & -h_{\rho_r}(\infty) = -\frac{1}{8} + \frac{1}{2} \log r.
\end{align*}    
Observe that 
\begin{gather*}
    -\frac{1}{8} + \frac{1}{2} \log r > \AZ{\rho_r}{\lambda} = \frac{1}{2} \log r - \frac{3}{8} + \frac{1}{2r^2}.
\end{gather*}
\end{example}

\subsection{Computations of Pairings} \label{computation of pairings}

In this section we give explicit computations for certain pairings $\AZ{\rho_r}{\mu}$, where $\mu$ is an adelic measure defined over $\Q$. We begin with the basic case where $\mu = \lambda$.

Note that by Example \ref{height of zeta} we have computed
    \begin{gather*}
        \AZ{\rho_r}{\lambda} = \begin{cases}
            \vspace{2mm} \displaystyle \frac{1}{8} - \frac{1}{2}\log r & \text{if } r\leq 1 \\ 
            \displaystyle \frac{1}{2} \log r - \frac{3}{8} + \frac{1}{2r^2} & \text{if } r>1.
        \end{cases}
    \end{gather*}
    Using calculus one may verify that $\AZ{\rho_r}{\lambda}$ is minimized when $r = \sqrt{2}$, and in this case $$\AZ{\rho_{\sqrt{2}}}{\lambda} = \frac{1}{4}\log 2 - \frac{1}{8} \approx 0.048287.$$

Now we will study $\AZ{\rho_r}{\mu}$, where $\mu$ is the adelic measure associated to the Chebyshev polynomial $T(x) = x^2 - 2$ (see \cite[\textsection 6]{frl.2} for the definition of the adelic measure associated to a rational map). We begin by computing $\AZ{\rho_1}{\mu}$. 

\begin{prop}
    Let $\mu$ be the adelic measure associated to the Chebyshev polynomial $T(x) = x^2 - 2$. Then 
    \begin{gather} \label{areal of radius 1 paired with chebyshev}
    \AZ{\rho_1}{\mu} = \frac{7}{24} - \frac{\sqrt{3}}{2\pi} + \frac{3\sqrt{3}}{4\pi} L(2 , \chi) \approx 0.339068,
\end{gather}
where $L(2,\chi)$ is the Dirichlet $L$-function associated to $\chi$, the nontrivial character modulo 3.
\end{prop}

\begin{proof}
Note that $T$ has good reduction everywhere, so $\mu_{v} = \lambda_v$ for every finite $v$. Also, $\mu_\infty$ is the equilibrium measure on $[-2,2]$, given by
\begin{gather*}
    d\mu_\infty = \frac{dx}{\pi \sqrt{4-x^2}}.
\end{gather*}

Now, since $\rho_{1,v} = \lambda_v = \mu_v$ at all finite $v$, we have 
\begin{align*}
    \AZ{\rho_1}{\mu} = & \sum_{v \in M_\Q} \frac{1}{2} (\rho_1 , \rho_1)_v - (\rho_1 , \mu)_v + \frac{1}{2} (\mu , \mu)_v \\
    = & h_{\rho_1} (\infty) - (\rho_1 , \mu)_\infty + \frac{1}{2} (\mu, \mu)_\infty.
\end{align*}
Since $\mu_\infty$ is the equilibrium measure of a set of capacity 1, this becomes
\begin{gather} \label{pairing of rho_r and mu}
    \AZ{\rho_1}{\mu} = h_{\rho_1} (\infty) - (\rho_1 , \mu)_\infty.
\end{gather}
By Theorem \ref{areal height formula} we know that $h_{\rho_1} = \nicefrac{1}{8}$, so we just need to compute
\begin{align} \label{local pairing of areal and chebyshev at infinity}
\begin{split}
    -(\rho_1, \mu)_\infty = & \iint \log \abs{z-w} d\rho_1 (w) d\mu (z) \\
   = & \int_{-2}^2 \frac{f_r (x) dx}{\pi \sqrt{4-x^2}} \\
   = & \int_{-2}^2 \frac{\log^+ \abs{x} dx}{\pi \sqrt{4-x^2}} + \int_{-1}^1 \frac{x^2-1}{2\pi \sqrt{4-x^2}} dx.
\end{split}
\end{align}
Note that by \cite[Proposition 16]{PST}, 
\begin{align*}
    \AZ{\mu}{\lambda} = & h_\mu (\infty) + \int \log^+ \abs{z} d\mu_\infty (z) \\
    = & \int_{-2}^2 \frac{\log^+ \abs{x} dx}{\pi \sqrt{4-x^2}}.
\end{align*}
But by \cite[Proposition 4.1]{bridy-larson}, 
\begin{gather*}
    \AZ{\mu}{\lambda} = \frac{3\sqrt{3}}{4\pi} L(2,\chi).
\end{gather*}
Thus (\ref{local pairing of areal and chebyshev at infinity}) reduces to
\begin{align} \label{local pairing of areal and chebyshev at infinity part 3}
    \begin{split}
            -(\rho_r, \mu)_\infty = & \frac{3\sqrt{3}}{4\pi} L(2,\chi) + \int_{-1}^1 \frac{x^2-1}{2\pi \sqrt{4-x^2}} dx \\
            = &  \frac{3\sqrt{3}}{4\pi} L(2,\chi) + \frac{1}{6} - \frac{\sqrt{3}}{2\pi}.
    \end{split}
\end{align}

Combining (\ref{pairing of rho_r and mu}) and (\ref{local pairing of areal and chebyshev at infinity part 3}), we obtain (\ref{areal of radius 1 paired with chebyshev}). 
\end{proof}

\begin{rmk}
    Note that $$\AZ{\mu}{\lambda} = \frac{3\sqrt{3}}{4\pi} L(2,\chi) = m(1+x+y),$$ where the latter equality was computed in \cite{smyth_on_mahler_measure}. For the areal Weil height, one should compare (\ref{areal of radius 1 paired with chebyshev}) with $$m_\D (1+x+y) = \frac{1}{6} - \frac{11\sqrt{3}}{16\pi} + \frac{3\sqrt{3}}{4\pi} L(2,\chi),$$ which was computed in \cite{lalin-roy-1}.
\end{rmk}

Now we will find the radius $r$ which minimizes $\AZ{\rho_r}{\mu}$. 

\begin{prop} \label{minimize pairing prop}
    Let $\mu$ be the adelic measure associated to the Chebyshev polynomial $T(x) = x^2 - 2$. Then the pairing $\AZ{\rho_r}{\mu}$ for $r \in (0,\infty)$ is minimized at $r=2$, at which
    \begin{gather*}
        \AZ{\rho_{\sqrt{2}}}{\mu} = \frac{1}{2} \log 2 - \frac{1}{8} \approx 0.221574.
    \end{gather*}
\end{prop}

\begin{proof}
Our strategy is to analyze $$\frac{d}{dr} \AZ{\rho_r}{\mu} = -\frac{1}{2r} + \frac{d}{dr} \int_{-2}^2 \frac{f_r (x)}{\pi \sqrt{4-x^2}} dx$$ by applying the Leibniz integral rule (e.g. \cite[Theorem 2.27]{folland}). Let $0 < \varepsilon < 2$ and define the function $$F_r (x) = \frac{f_r (x)}{\pi \sqrt{4-x^2}}.$$ Observe that for any $r \in (0,\infty)$, $F_r (x)$ is integrable on $(-2,2)$ and its partial derivative $$\frac{\partial}{\partial r} F_r (x) = \begin{cases}
    0 & \text{if } r \leq x,\\
    \displaystyle \frac{r^2 - x^2}{r^3 \pi \sqrt{4-x^2}} & \text{if } r \geq x
\end{cases}$$ exists for every $x \in (-2,2)$. Furthermore, for $r \geq \varepsilon$ we have
\begin{gather*}
    0 \leq \frac{\partial}{\partial r} F_r (x) \leq \left ( \frac{1}{\varepsilon} + \frac{4}{\varepsilon^3} \right ) \frac{1}{r^3 \pi \sqrt{4-x^2}}
\end{gather*}
for all $x \in (-2,2)$. Since the right-hand side of the above inequality is integrable, by the Leibniz integral rule we have
\begin{align} \label{leibniz rule}
    \begin{split}
    \frac{d}{dr} \AZ{\rho_r}{\mu} = & -\frac{1}{2r} + \int_{-2}^2 \frac{d}{dr} F_r (x) dx \\
    = & -\frac{1}{2r} + \int_{-\min \{2,r\}}^{\min \{2,r\}}  \frac{1}{\pi \sqrt{4-x^2}} \left ( \frac{1}{r} - \frac{x^2}{r^3} \right ) dx.
    \end{split}
\end{align}
for every $r \in (\varepsilon, \infty)$. By letting $\varepsilon$ be arbitrarily small, we see this holds for every $r \in (0,\infty)$. If $r \leq 2$, (\ref{leibniz rule}) becomes
\begin{gather*}
    \frac{d}{dr} \AZ{\rho_r}{\mu} = -\frac{1}{2r} + \frac{2}{\pi r} \sin^{-1} \left ( \frac{r}{2} \right ) + \frac{1}{\pi r^2} \sqrt{4-r^2} - \frac{4}{\pi r^3} \sin^{-1} \left ( \frac{r}{2} \right ),
\end{gather*}
which is negative when $r<2$ and zero at $r=2$. If $r > 2$ then (\ref{leibniz rule}) evaluates to
\begin{gather*}
     \frac{d}{dr} \AZ{\rho_r}{\mu} = \frac{r^2-4}{2r^3},
\end{gather*}
which is positive. Thus by calculus $\AZ{\rho_r}{\mu}$ is minimized at $r = 2$, at which
\begin{align*}
    \AZ{\rho_2}{\mu} = & \frac{1}{8} - \frac{1}{2}\log 2 + \int_{-2}^2 \frac{1}{\pi \sqrt{4-x^2}} \left ( \log 2 - \frac{1}{2} + \frac{x^2}{8} \right ) dx \\
    = & \frac{1}{2}\log 2 - \frac{1}{8}. \qedhere
\end{align*}
\end{proof}

\begin{rmk}
    An interesting identity is $\AZ{\rho_2}{\mu} = - h_{\rho_2} (\infty)$, which also happens to be the quantity computed in Example \ref{uniform equidistribution example}. 
\end{rmk}

\bibliographystyle{amsalpha}
\bibliography{references}

@article {pritsker-areal,
    AUTHOR = {Pritsker, Igor E.},
     TITLE = {An areal analog of {M}ahler's measure},
   JOURNAL = {Illinois J. Math.},
  FJOURNAL = {Illinois Journal of Mathematics},
    VOLUME = {52},
      YEAR = {2008},
    NUMBER = {2},
     PAGES = {347--363},
      ISSN = {0019-2082,1945-6581},
   MRCLASS = {11R06 (11C08 11G50 30C10)},
  MRNUMBER = {2524641},
MRREVIEWER = {Lenny\ Fukshansky},
       URL = {http://projecteuclid.org/euclid.ijm/1248355339},
}

@book {bombieri-gubler,
    AUTHOR = {Bombieri, Enrico and Gubler, Walter},
     TITLE = {Heights in {D}iophantine geometry},
    SERIES = {New Mathematical Monographs},
    VOLUME = {4},
 PUBLISHER = {Cambridge University Press, Cambridge},
      YEAR = {2006},
     PAGES = {xvi+652},
      ISBN = {978-0-521-84615-8; 0-521-84615-3},
   MRCLASS = {11G50 (11-02 11G10 11G30 11J68 14G40)},
  MRNUMBER = {2216774},
MRREVIEWER = {Yuri\ Bilu},
       DOI = {10.1017/CBO9780511542879},
       URL = {https://doi.org/10.1017/CBO9780511542879},
}

@book {baker-rumely.book,
    AUTHOR = {Baker, Matthew and Rumely, Robert},
     TITLE = {Potential theory and dynamics on the {B}erkovich projective
              line},
    SERIES = {Mathematical Surveys and Monographs},
    VOLUME = {159},
 PUBLISHER = {American Mathematical Society, Providence, RI},
      YEAR = {2010},
     PAGES = {xxxiv+428},
      ISBN = {978-0-8218-4924-8},
   MRCLASS = {37P50 (14G20 14G22 31C15 31C45 37P40)},
  MRNUMBER = {2599526},
MRREVIEWER = {Charles\ Favre},
       DOI = {10.1090/surv/159},
       URL = {https://doi.org/10.1090/surv/159},
}

@article {frl.2,
    AUTHOR = {Favre, Charles and Rivera-Letelier, Juan},
     TITLE = {\'Equidistribution quantitative des points de petite hauteur
              sur la droite projective},
   JOURNAL = {Math. Ann.},
  FJOURNAL = {Mathematische Annalen},
    VOLUME = {335},
      YEAR = {2006},
    NUMBER = {2},
     PAGES = {311--361},
      ISSN = {0025-5831,1432-1807},
   MRCLASS = {11G50 (37F10)},
  MRNUMBER = {2221116},
MRREVIEWER = {Matthew\ H.\ Baker},
       DOI = {10.1007/s00208-006-0751-x},
       URL = {https://doi.org/10.1007/s00208-006-0751-x},
}

@article{lehmer,
 ISSN = {0003486X, 19398980},
 URL = {http://www.jstor.org/stable/1968172},
 author = {D. H. Lehmer},
 journal = {Annals of Mathematics},
 number = {3},
 pages = {461--479},
 publisher = {[Annals of Mathematics, Trustees of Princeton University on Behalf of the Annals of Mathematics, Mathematics Department, Princeton University]},
 title = {Factorization of Certain Cyclotomic Functions},
 urldate = {2025-03-25},
 volume = {34},
 year = {1933}
}

@article {kronecker,
    AUTHOR = {Kronecker, L.},
     TITLE = {Zwei {S}\"atze \"uber {G}leichungen mit ganzzahligen
              {C}oefficienten},
   JOURNAL = {J. Reine Angew. Math.},
  FJOURNAL = {Journal f\"ur die Reine und Angewandte Mathematik. [Crelle's
              Journal]},
    VOLUME = {53},
      YEAR = {1857},
     PAGES = {173--175},
      ISSN = {0075-4102,1435-5345},
   MRCLASS = {99-04},
  MRNUMBER = {1578994},
       DOI = {10.1515/crll.1857.53.173},
       URL = {https://doi.org/10.1515/crll.1857.53.173},
}

@incollection {smyth-survey,
    AUTHOR = {Smyth, Chris},
     TITLE = {The {M}ahler measure of algebraic numbers: a survey},
 BOOKTITLE = {Number theory and polynomials},
    SERIES = {London Math. Soc. Lecture Note Ser.},
    VOLUME = {352},
     PAGES = {322--349},
 PUBLISHER = {Cambridge Univ. Press, Cambridge},
      YEAR = {2008},
      ISBN = {978-0-521-71467-9},
   MRCLASS = {11R06},
  MRNUMBER = {2428530},
       DOI = {10.1017/CBO9780511721274.021},
       URL = {https://doi.org/10.1017/CBO9780511721274.021},
}

@book {around-the-unit-circle,
    AUTHOR = {McKee, James and Smyth, Chris},
     TITLE = {Around the unit circle---{M}ahler measure, integer matrices
              and roots of unity},
    SERIES = {Universitext},
 PUBLISHER = {Springer, Cham},
      YEAR = {[2021] \copyright 2021},
     PAGES = {xx+438},
      ISBN = {978-3-030-80030-7; 978-3-030-80031-4},
   MRCLASS = {11R06 (05A05 05C20 05C22 11-02 11C20 15B36)},
  MRNUMBER = {4393584},
MRREVIEWER = {Art\=uras\ Dubickas},
       DOI = {10.1007/978-3-030-80031-4},
       URL = {https://doi.org/10.1007/978-3-030-80031-4},
}

@article {bilu,
    AUTHOR = {Bilu, Yuri},
     TITLE = {Limit distribution of small points on algebraic tori},
   JOURNAL = {Duke Math. J.},
  FJOURNAL = {Duke Mathematical Journal},
    VOLUME = {89},
      YEAR = {1997},
    NUMBER = {3},
     PAGES = {465--476},
      ISSN = {0012-7094,1547-7398},
   MRCLASS = {11G35 (11G25 11J68 14G05 14G25)},
  MRNUMBER = {1470340},
MRREVIEWER = {Dan\ Abramovich},
       DOI = {10.1215/S0012-7094-97-08921-3},
       URL = {https://doi.org/10.1215/S0012-7094-97-08921-3},
}

@misc{fili_ui,
      title={A metric of mutual energy and unlikely intersections for dynamical systems}, 
      author={Paul Fili},
      year={2017},
      eprint={1708.08403},
      archivePrefix={arXiv},
      primaryClass={math.NT},
      url={https://arxiv.org/abs/1708.08403}, 
}

@misc{orloski-sardari,
      title={Limiting distributions of conjugate algebraic integers}, 
      author={Bryce Joseph Orloski and Naser Talebizadeh Sardari},
      year={2024},
      eprint={2302.02872},
      archivePrefix={arXiv},
      primaryClass={math.NT},
      url={https://arxiv.org/abs/2302.02872}, 
}

@article {choi-samuels,
    AUTHOR = {Choi, Kwok-Kwong Stephen and Samuels, Charles L.},
     TITLE = {Two inequalities on the areal {M}ahler measure},
   JOURNAL = {Illinois J. Math.},
  FJOURNAL = {Illinois Journal of Mathematics},
    VOLUME = {56},
      YEAR = {2012},
    NUMBER = {3},
     PAGES = {825--834},
      ISSN = {0019-2082,1945-6581},
   MRCLASS = {11C08 (11G50)},
  MRNUMBER = {3161353},
MRREVIEWER = {Rupam\ Barman},
       URL = {http://projecteuclid.org/euclid.ijm/1391178550},
}

@article {flammang,
    AUTHOR = {Flammang, V.},
     TITLE = {The {M}ahler measure and its areal analog for totally positive
              algebraic integers},
   JOURNAL = {J. Number Theory},
  FJOURNAL = {Journal of Number Theory},
    VOLUME = {151},
      YEAR = {2015},
     PAGES = {211--222},
      ISSN = {0022-314X,1096-1658},
   MRCLASS = {11R06},
  MRNUMBER = {3314210},
MRREVIEWER = {Detchat\ Samart},
       DOI = {10.1016/j.jnt.2014.12.023},
       URL = {https://doi.org/10.1016/j.jnt.2014.12.023},
}

@article {lalin-roy-1,
    AUTHOR = {Lalin, Matilde N. and Roy, Subham},
     TITLE = {Evaluations of the areal {M}ahler measure of multivariable
              polynomials},
   JOURNAL = {J. Number Theory},
  FJOURNAL = {Journal of Number Theory},
    VOLUME = {254},
      YEAR = {2024},
     PAGES = {103--145},
      ISSN = {0022-314X,1096-1658},
   MRCLASS = {11R06 (11M06 11R42 30C10)},
  MRNUMBER = {4636754},
MRREVIEWER = {James\ McKee},
       DOI = {10.1016/j.jnt.2023.07.006},
       URL = {https://doi.org/10.1016/j.jnt.2023.07.006},
}

@article {lalin-roy-2,
    AUTHOR = {Lalin, Matilde N. and Roy, Subham},
     TITLE = {The areal {M}ahler measure under a power change of variables},
   JOURNAL = {Illinois J. Math.},
  FJOURNAL = {Illinois Journal of Mathematics},
    VOLUME = {68},
      YEAR = {2024},
    NUMBER = {2},
     PAGES = {309--330},
      ISSN = {0019-2082,1945-6581},
   MRCLASS = {11R06 (11M06 11R42 32A99)},
  MRNUMBER = {4760277},
       DOI = {10.1215/00192082-11161282},
       URL = {https://doi.org/10.1215/00192082-11161282},
}

@book {benedetto,
    AUTHOR = {Benedetto, Robert L.},
     TITLE = {Dynamics in one non-archimedean variable},
    SERIES = {Graduate Studies in Mathematics},
    VOLUME = {198},
 PUBLISHER = {American Mathematical Society, Providence, RI},
      YEAR = {2019},
     PAGES = {xviii+463},
      ISBN = {978-1-4704-4688-8},
   MRCLASS = {37-01 (11S82 37Pxx)},
  MRNUMBER = {3890051},
MRREVIEWER = {Michael\ Louis\ Tepper},
       DOI = {10.1090/gsm/198},
       URL = {https://doi.org/10.1090/gsm/198},
}

@book {berkovich,
    AUTHOR = {Berkovich, Vladimir G.},
     TITLE = {Spectral theory and analytic geometry over non-{A}rchimedean
              fields},
    SERIES = {Mathematical Surveys and Monographs},
    VOLUME = {33},
 PUBLISHER = {American Mathematical Society, Providence, RI},
      YEAR = {1990},
     PAGES = {x+169},
      ISBN = {0-8218-1534-2},
   MRCLASS = {32P05 (32C15 32C37 46S10 47S10)},
  MRNUMBER = {1070709},
MRREVIEWER = {W.\ Bartenwerfer},
       DOI = {10.1090/surv/033},
       URL = {https://doi.org/10.1090/surv/033},
}

@article {PST,
    AUTHOR = {Petsche, Clayton and Szpiro, Lucien and Tucker, Thomas J.},
     TITLE = {A dynamical pairing between two rational maps},
   JOURNAL = {Trans. Amer. Math. Soc.},
  FJOURNAL = {Transactions of the American Mathematical Society},
    VOLUME = {364},
      YEAR = {2012},
    NUMBER = {4},
     PAGES = {1687--1710},
      ISSN = {0002-9947,1088-6850},
   MRCLASS = {37P30 (11G50 14G40 37P05)},
  MRNUMBER = {2869188},
MRREVIEWER = {Kevin\ Doerksen},
       DOI = {10.1090/S0002-9947-2011-05350-X},
       URL = {https://doi-org.argo.library.okstate.edu/10.1090/S0002-9947-2011-05350-X},
}

@article{lalin-nair-ringeling-roy,
    author = {Lalín, Matilde N and Nair, Siva Sankar and Ringeling, Berend and Roy, Subham},
    title = {Random Walks Through the Areal Mahler Measure: Steps in the Complex Plane},
    journal = {The Quarterly Journal of Mathematics},
    pages = {haaf043},
    year = {2025},
    month = {11},
    issn = {0033-5606},
    doi = {10.1093/qmath/haaf043},
    url = {https://doi.org/10.1093/qmath/haaf043},
    eprint = {https://academic.oup.com/qjmath/advance-article-pdf/doi/10.1093/qmath/haaf043/65403933/haaf043.pdf},
}

@book {cassels,
    AUTHOR = {Cassels, J. W. S.},
     TITLE = {Local fields},
    SERIES = {London Mathematical Society Student Texts},
    VOLUME = {3},
 PUBLISHER = {Cambridge University Press, Cambridge},
      YEAR = {1986},
     PAGES = {xiv+360},
      ISBN = {0-521-30484-9; 0-521-31525-5},
   MRCLASS = {11Sxx (11D88 11E95 11F85 14G20)},
  MRNUMBER = {861410},
MRREVIEWER = {R.\ A.\ Mollin},
       DOI = {10.1017/CBO9781139171885},
       URL = {https://doi-org.argo.library.okstate.edu/10.1017/CBO9781139171885},
}

@article {bridy-larson,
    AUTHOR = {Bridy, Andrew and Larson, Matt},
     TITLE = {The {A}rakelov-{Z}hang pairing and {J}ulia sets},
   JOURNAL = {Proc. Amer. Math. Soc.},
  FJOURNAL = {Proceedings of the American Mathematical Society},
    VOLUME = {149},
      YEAR = {2021},
    NUMBER = {9},
     PAGES = {3699--3713},
      ISSN = {0002-9939,1088-6826},
   MRCLASS = {37P15 (11G50 14G40)},
  MRNUMBER = {4291571},
MRREVIEWER = {Hongming\ Nie},
       DOI = {10.1090/proc/15518},
       URL = {https://doi.org/10.1090/proc/15518},
}

@article {smyth_on_mahler_measure,
    AUTHOR = {Smyth, C. J.},
     TITLE = {On measures of polynomials in several variables},
   JOURNAL = {Bull. Austral. Math. Soc.},
  FJOURNAL = {Bulletin of the Australian Mathematical Society},
    VOLUME = {23},
      YEAR = {1981},
    NUMBER = {1},
     PAGES = {49--63},
      ISSN = {0004-9727},
   MRCLASS = {10K50 (12A15)},
  MRNUMBER = {615132},
MRREVIEWER = {G\'erard\ Rauzy},
       DOI = {10.1017/S0004972700006894},
       URL = {https://doi.org/10.1017/S0004972700006894},
}

@book {folland,
    AUTHOR = {Folland, Gerald B.},
     TITLE = {Real analysis},
    SERIES = {Pure and Applied Mathematics (New York)},
   EDITION = {Second},
      NOTE = {Modern techniques and their applications,
              A Wiley-Interscience Publication},
 PUBLISHER = {John Wiley \& Sons, Inc., New York},
      YEAR = {1999},
     PAGES = {xvi+386},
      ISBN = {0-471-31716-0},
   MRCLASS = {00A05 (26-01 28-01 46-01)},
  MRNUMBER = {1681462},
}

\end{document}